\documentclass[12pt]{article}
\usepackage{amsmath,amsfonts,amsthm,amscd, amssymb}

\usepackage[all]{xy}

\theoremstyle{plain}
\newtheorem{thm}{Theorem}[section]
\newtheorem{pro}[thm]{Proposition}
\newtheorem{cor}[thm]{Corollary}
\newtheorem{lem}[thm]{Lemma}

\theoremstyle{definition}
\newtheorem{ex}[thm]{Example}
\newtheorem{rem}[thm]{Remark}
\newtheorem{defn}[thm]{Definition}
 \newtheorem*{acknowledgement}{Acknowledgement}

\newcommand{\Tor}{\operatorname{Tor}}

\newcommand{\res}{\operatorname{res}}

\newcommand{\End}{\operatorname{End}}
\newcommand{\Hom}{\operatorname{Hom}}
\newcommand{\Sym}{\operatorname{Sym}}

\newcommand{\Img}{\operatorname{Im}}
\newcommand{\coker}{\operatorname{coker}}

\newcommand{\id}{\operatorname{id}}

\newcommand{\FF}{{\mathbb{F}}}

\newcommand{\ZZ}{{\mathbb{Z}}}

\newcommand{\gl}{{\mathfrak{gl}}}
\newcommand{\g}{{\mathfrak{g}}}

\def\div{\raise 1pt \hbox{\big|}}

\newcommand{\vv}{\, | \,}

\newcommand{\iso}{\cong}

\def\maprt#1{\smash{\,\mathop{\longrightarrow}\limits^{#1}\,}}

\begin{document}

\title{Bockstein Closed 2-Group Extensions and Cohomology of Quadratic Maps}
\author{Jonathan Pakianathan and Erg\"un Yal\c c\i n}
\maketitle

\begin{abstract}
A central extension of the form $E: 0 \to V \to G \to W \to 0$,
where $V$ and $W$ are elementary abelian $2$-groups, is called
Bockstein closed if the components  $ q_i \in H^*(W, \FF _2 )$ of
the extension class of $E$ generate an ideal which is closed under
the Bockstein operator. In this paper, we study the cohomology ring
of $G$ when $E$ is a Bockstein closed $2$-power exact extension. The
mod-$2$ cohomology ring of $G$ has a simple form and it is easy to
calculate. The main result of the paper is the calculation of the
Bocksteins of the generators of the mod-$2$ cohomology ring using an
Eilenberg-Moore spectral sequence. We also find an interpretation of
the second page of the Bockstein spectral sequence in terms of a new
cohomology theory that we define for Bockstein closed quadratic maps
$Q : W \to V$ associated to the extensions $E$ of the above form.

\smallskip

\end{abstract}

\section{Introduction}\label{sect:intro}

Let $G$ be a $p$-group which fits into a central extension of the
form
$$ E: 0\to V \to G \to W \to 0$$
where $V$, $W$ are $\FF_p$-vector spaces of dimensions $n$ and $m$,
respectively. $E$ is called Bockstein closed if the components  $
q_i \in H^*(W, \FF _p )$ of the extension class of $E$ generate an
ideal which is closed under the Bockstein operator. We say $E$ is
$p$-power exact if the following three conditions are satisfied: (i)
$m=n$, (ii) $V$ is the Frattini subgroup of $G$, and (iii) the
$p$-rank of $G$ is equal to $n$. Associated to $G$ there is a $p$-th
power map $(\ )^p :W\to V$. When $p$ is odd, the $p$-th power map is
a homomorphism and if $E$ is also $p$-power exact, then it is an
isomorphism. Using this isomorphism, one can define a Bracket $[\ ,
\ ]: W \times W \to W$ on $W$ which turns out to be a Lie Bracket if
and only if the associated extension is Bockstein closed. This was
studied by Browder-Pakianathan \cite{BrPa} who also used this fact
to give a complete description of the Bockstein cohomology of $G$ in
terms of the Lie algebra cohomology of the associated Lie algebra.
This theory was later used by Pakianathan \cite{Pak} to give a
counterexample to a conjecture of Adem \cite{Adem} on exponents in the integral
cohomology of $p$-groups for odd primes $p$.

In the case where $p=2$, the $2$-power map $(\ )^2 :W \to V$ is not
a homomorphism, so the results of Browder-Pakianathan do not
generalize to $2$-groups in a natural way. In this case, the
$2$-power map is a quadratic map $Q: W \to V$ where the associated
bilinear map $B: W \times W  \to V$ is induced by taking commutators
in $G$. The $2$-power exact condition is equivalent to the
conditions: (i) $m=n$, (ii) the elements $\{ Q(w)\vv w\in W\}$
generate $V$, and (iii) if $Q(w)=0$ for some $w\in W$, then $w=0$.
We studied the quadratic maps associated to Bockstein closed
extensions in an earlier paper, and showed that an extension $E$ is
Bockstein closed if and only if there is a bilinear map $P: V\times
W \to V$ such that
\begin{equation}\label{eqn:P-relation}
P(Q(w), w')= B(w,w')+P(B(w,w'), w)
\end{equation}
holds for all $w,w' \in W$ (see Theorem 1.1 in \cite{PaYa}). In some
sense this is the Jacobi identity for the $p=2$ case. If there is a
quadratic map $Q: W\to W$ which satisfies this identity with $P=B$,
then the vector space $W$ becomes a $2$-restricted Lie algebra with
$2$-power map defined by $w^{[2]}=Q(w)+w$ for all $w \in W$. But in
general there are no direct connections between Bockstein closed
quadratic maps and mod-$2$ Lie algebras.

In this paper, we study the cohomology of Bockstein closed $2$-power
exact extensions. We calculate the mod-$2$ cohomology ring and give
a description of the Bockstein spectral sequence. As in the case
when $p$ is odd, the Bockstein spectral sequence can be described in
terms of a cohomology theory based on our algebraic data. In this
case, the right cohomology theory is the cohomology $H^*(Q,U)$ of a
Bockstein closed quadratic map $Q: W \to V$. We define this
cohomology using an explicit cochain complex associated to the
quadratic map. The definition is given in such a way that the low
dimensional cohomology has interpretation in terms of extensions of
Bockstein closed quadratic maps. For example, $H^0(Q,U)$ gives the
$Q$-invariants of $U$ and $H^1(Q, U)\iso \Hom_{\bf Quad} (Q, U)$ if
$U$ is a trivial $Q$-module. Also, $H^2(Q,U)$ is isomorphic to the
group of extensions of $Q$ with abelian kernel $U$ (see Proposition
\ref{pro:H^2}). The definition of $H^*(Q,U)$ is given in Section
\ref{sect:cohdefn}. To keep the theory more general, in the
definition of $H^*(Q,U)$ we do not assume that the quadratic map $Q$
is $2$-power exact.

In Section \ref{sect:mod2cohomology}, we calculate the mod-$2$
cohomology ring of a Bockstein closed $2$-power exact group $G$
using the Lyndon-Hochschild-Serre spectral sequence associated to
the extension $E$. This calculation is relatively easy and it is
probably known to experts in the field (see, for example,
\cite{MiSy} or \cite{Rusin}). The mod-$2$ cohomology ring of $G$ has
a very nice expression given by
$$H^*(G, \FF_2)\cong A^*(Q) \otimes \FF_2 [s_1,...,s_n]$$
where $s_i$'s are some two dimensional generators and the algebra
$A^*(Q)$ is given by
$$A^* (Q)= \FF _2 [x_1,\dots, x_n]/(q_1, \dots, q_n)$$
where $\{x_1, \dots, x_n\}$ forms a basis for $H^1(W)$ and $q_i$'s
are components of the extension class $q \in H^2 (W, V)$ with
respect to a basis for $V$. The action of the Bockstein operator on
this cohomology ring gives valuable information about the question
of whether the extension can be uniformly lifted to other
extensions. Also finding Bocksteins of generators of the mod-$2$
cohomology algebra is the starting point for calculating the
integral cohomology of $G$. We prove the following:

\begin{thm}\label{thm:R=L} Let  $E: 0 \to V \to G  \to W \to 0 $
be a Bockstein closed $2$-power exact extension with extension class
$q$ and let $\beta (q)=L q$. Then the mod-$2$ cohomology of $G$ is
in the above form and $\beta (s) =Ls+\eta$ where $s$ denotes a
column matrix with entries in $s_i$'s and  $\eta$ is a column matrix
with entries in $H^3 (W, \FF _2)$.
\end{thm}

The proof of this theorem is given in Section \ref{sect:proof of the
main theorem} using the Eilenberg-Moore spectral sequence associated
to the extension. The key property of the EM-spectral sequence is
that it behaves well under the Steenrod operations. The Steenrod
algebra structure of the EM-spectral sequence was studied by L.
Smith \cite{Smith1}, \cite{Smith2} and D. Rector \cite{Rector}
independently in a sequence of papers. Here we use only a special
case of these results. More precisely, we use the fact that the
first two vertical lines in the EM-spectral sequence are closed
under the action of the Steenrod algebra. This is stated as
Corollary 4.4 in \cite{Smith1}.

The column matrix $\eta$ of the formula $\beta (s) =Ls+\eta$ defines
a cohomology class $[\eta] \in H^3(Q, L)$ where $L$ is the
$Q$-module associated to the matrix $L$. Recall that in the work of
Browder-Pakianathan \cite{BrPa}, there is a cohomology class lying
in the Lie algebra cohomology $H^3 ({\cal L}, {\rm ad})$ which is
defined in a similar way and it is an obstruction class for lifting
$G$ uniformly twice. We obtain a similar theorem for uniform double
lifting of $2$-group extensions.

\begin{thm}
\label{thm:doublelifting} Let  $E: 0 \to V \to G\to W \to 0 $ be a
Bockstein closed $2$-power exact extension with extension class $q$.
Let $Q$ be the associated quadratic map and $L$ denote the
$Q$-module defined by $L$ in the equation $\beta (q)=Lq$. Then, the
extension $E$ has a uniform double lifting if and only if $[\eta]=0$
in $H^3(Q, L)$.
\end{thm}

Another result we have is a description of the second page of the
Bockstein spectral sequence in terms of the cohomology of Bockstein
closed quadratic maps for the case where the extension has a uniform
double lifting.

\begin{thm} Let  $E$, $G$, $Q$, and $L$ be as in Theorem \ref{thm:doublelifting}.
Assume that $E$ has a uniform double lifting. Then, the second page
of the Bockstein spectral sequence for $G$ is given by
$$B_2 ^*(G)=  \bigoplus _{i=0} ^{\infty} H^{*-2i} (Q, \Sym^i (L))$$
where $\Sym^i(L)$ denotes the symmetric $i$-th power of $L$.
\end{thm}

In the $p$ odd case, the $B$-cohomology has been calculated in cases
by comparing it to $H^*(\g,U(\g)^*)$ where $U(\g)^*$ is the dual of
the universal enveloping algebra of $\g$ equipped with the dual
adjoint action, where $\g$ is an associated complex Lie algebra.
This fundamental object has played a role in string topology
(homology of free loop spaces) and is analogous to the (classical)
ring of modular forms, identified by Eichler-Shimura as
$H^*(SL_2(\mathbb{Z}), \text{Poly} (V))$ where $V$ is the complex
2-dimensional canonical representation of $SL_2(\mathbb{Z})$ (see
\cite{PaRo} for more details). In string topology contexts, this is
referred to as the Hodge decomposition and so the above can be
thought of as a Hodge decomposition for the quadratic form $Q$. It
describes the distribution of higher torsion in the integral
cohomology of the associated group $G$.

As in the case of Lie algebras, it is possible to give a suitable
definition of a universal enveloping algebra $U(Q)$ for a quadratic
map $Q$ so that the representations of $Q$ and representations of
the universal algebra $U(Q)$ can be identified in a natural way.
However it is not clear to us how to find an isomorphism between the
cohomology of the universal algebra $U(Q)$ and the cohomology of the
quadratic map $Q$. There is also the issue of finding analogies of
the theorems on universal algebras of Lie algebras such as the
Poincar{\' e}-Birkoff-Witt Theorem. We leave these as open problems.

The paper is organized as follows: In Section \ref{sect:categories},
we introduce the category of quadratic maps and show that it is
naturally equivalent to the category of extensions of certain type.
Then in Section \ref{sect:Bclosedmaps}, we give the definition of a
Bockstein closed quadratic map. The definition of cohomology of
Bockstein closed quadratic maps is given in Section
\ref{sect:cohdefn}. Sections \ref{sect:mod2cohomology},
\ref{section: EMSS}, and \ref{sect:proof of the main theorem} are
devoted to the mod-$2$ cohomology calculations using LHS- and
EM-spectral sequences and the calculation of Bocksteins of the
generators. In particular, Theorem \ref{thm:R=L} is proven in
Section \ref{sect:proof of the main theorem}. In Section
\ref{sect:uniformdoublelifting}, we discuss the obstructions for
uniform lifting and in Section \ref{sect:BockSpectralSeq}, we
explain the $E_2$-page of the Bockstein spectral sequence in terms
of the cohomology of Bockstein closed quadratic maps.

\section{Category of quadratic maps}
\label{sect:categories}

Let $E$ denote a central extension of the form $$E:0 \to V \to G \to
W\to 0$$ where $V$ and $W$ are elementary abelian $2$-groups.
Associated to $E$, there is a cohomology class $q \in H^2 (W, V)$.
Also associated to $E$ there is a quadratic map $Q: W \to V$ defined
by $Q(w)=(\hat w)^2$, where $\hat w$ denotes an element in $G$ that
lifts $w \in W$. Similarly, the commutator induces a symmetric
bilinear map $B: W \times W \to V$ defined by $B(x,y)=[\hat x, \hat
y]$ for $x,y\in W$ where $[g,h]=g^{-1}h^{-1}gh$ for $g,h\in G$. It
is easy to see that $B$ is the bilinear form associated to $Q$.

We have shown in \cite{PaYa} that the extension class $q$ and the
quadratic form $Q$ are closely related to each other. In particular,
we showed that we can take $q=[f]$ where $f$ is a bilinear factor
set $f: W \times W\to V$ satisfying the identity $f(w,w)=Q(w)$ for
all $w \in W$ (see \cite[Lemma 2.3]{PaYa}). We can write this
correspondence more explicitly by choosing a basis $\{ w_1,\dots,
w_m \}$ for $W$. Then,
\[f(w_i, w_j)=
\begin{cases}
B( w_i , w_j )  & \text{if \: \: $i < j$}, \\
Q(w_i)  &  \text{if \: \: $i=j$}, \\
0 \ \  & \text{if \: \: $i>j$}.
\end{cases}\]
This gives a very specific expression for $q$. Let $\{v_1,\dots,
v_n\}$ be a basis for $V$, and let $q_k$ be the $k$-th component of
$q$ with respect to this basis. Then, $$q_k =\sum _i Q_k (w_i) x_i
^2 + \sum _{i<j} B_k (w_i, w_j ) x_i x_j$$ where $\{ x_1,\dots ,
x_m\}$ is the dual basis of $\{ w_1,\dots, w_m \}$ and $Q_k$ and
$B_k$ denote the $k$-th components of $Q$ and $B$. This allows one
to prove the following:

\begin{pro}[Corollary 2.4 in \cite{PaYa}]
Given a quadratic map $Q: W \to V$, there is a unique (up to
equivalence) central extension $$E(Q): 0 \to V \to G(Q) \to W \to
0$$ with a bilinear factor set $f:W \times W \to V$ satisfying
$f(w,w)=Q(w)$ for all  $w \in W$.
\end{pro}

This gives a bijective correspondence between quadratic maps $Q: W
\to V$ and the central extensions of the form $E : 0 \to V \to G \to
W \to 0$. We will now define the category of quadratic maps and the
category of group extensions of the above type and then prove that
the correspondence described above indeed gives a natural
equivalence between these categories.

\subsection{Equivalence of categories}
\label{subsect:equivalence}

The category  of quadratic maps $\bf Quad$ is defined as the
category whose objects are quadratic maps $Q : W \to V$ where $W$
and $V$ are vector spaces over $\FF _2$. For quadratic maps $Q_1,
Q_2$, a morphism $f : Q_1 \to Q_2$ is defined as a pair of linear
transformations $f=(f_W, f_V)$ such that the following diagram
commutes:
\begin{equation}\label{eqn:quadmorp}
\begin{CD}
W_1 @>f_W>> W_2 \\
@VVQ_1V @VVQ_2V \\
V_1 @>f_V>>\  V_2\, .
\end{CD}
\end{equation}
The composition of morphisms $f=(f_W, f_V)$ and $g=(g_W, g_V)$ is
defined by coordinate-wise compositions. The identity morphism is
the pair $(\id_W, \id_V)$. Two quadratic maps $Q_1$ and $Q_2$ are
isomorphic if there are morphisms $f: Q_1 \to Q_2$ and $g : Q_2 \to
Q_1 $ such that $f\circ g= \id_{Q_2}$ and $g \circ f =
\id_{Q_1}$.

The category $\bf Ext$ is defined as the category whose objects are
the equivalence classes of extensions of type
$$E: 0 \to V \to G \to W \to 0 $$
where $V$ and $W$ are vector spaces over $\FF _2$, and the morphisms
are given by a commuting diagram as follows:
\begin{equation}
\label{eqn:extmorp}
\begin{CD}
E_1: @. \ \ \  0 @>>> V_1 @>>> G_1 @>>> W_1 @>>> 0 \\
@VVfV @. @VVf_{V}V @VVf_GV @VVf_WV @. \\
E_2: @. \ \ \  0 @>>> V_2 @>>> G_2 @>>> W_2 @>>>\ 0\, . \\
\end{CD}
\end{equation}
Note that two extensions are considered equivalent if there is a
diagram as above with $f_W=\id_W$ and $f_V=\id_V$. All such
morphisms are taken to be equal to identity morphism in our
category. More generally, two morphisms $f,g: E_1 \to E_2$ will be
considered equal in $\bf Ext$ if $f_V=g_V$ and $f_W=g_W$.

From the discussion at the beginning of this section, it is clear
that the assignments $\Phi : Q \to E(Q)$ and $\Psi: E \to Q_E$ give
a bijective correspondence between the objects of $\bf Quad$ and
$\bf Ext$. We just need to extend this correspondence to a
correspondence between morphisms. Given a morphism $f:E_1 \to E_2$,
we take $\Psi (f)$ to be the pair $(f_V, f_W):Q_1 \to Q_2$. By
commutativity of the diagram (\ref{eqn:extmorp}), it is easy to see
that $Q_2(f_W (w))=f_V(Q_1(w))$ holds for all $w \in W_1$. To define
the image of a morphism $f: Q_1 \to Q_2$ under $\Phi$, we need to
define a group homomorphism $f_G : G(Q_1) \to G(Q_2)$ which makes
the diagram given in (\ref{eqn:extmorp}) commute. Note that once $f_G$ is defined, we can define the morphism $\Phi (f): E_1
\to E_2$ as a sequence of maps $(f_V, f_G, f_W)$ as in 
diagram (\ref{eqn:extmorp}). It is clear that the composition $\Psi \circ \Phi$ is equal to the identity transformation. The composition $\Phi \circ \Psi$ is also equal to the identity in $\bf Ext$ although it may not be equal to identity on the middle map $f_G$. This follows from the fact that two morphisms $f,g: E_1 \to E_2$ between two extensions are equal in $\bf Ext$ if $f_V=g_V$ and $f_W=g_W$.

To define a group homomorphism $f_G : G(Q_1) \to G(Q_2)$ which makes the diagram given in (\ref{eqn:extmorp}) commute, first recall
that for $i=1,2$, we can take $G(Q_i)$ as the set $V_i\times W_i$ with
multiplication given by $(v,w)(v',w')=(v+v'+f_i(w,w'), w+w')$ where
$f_i:W_i\times W_i \to V_i$ is a bilinear factor set satisfying $f_i (w,w)=Q_i(w)$ for every $w\in W_i$ (see \cite[Lem. 2.3]{PaYa}). Note that the choice of the factor set is not unique and if $f_i$ and $f_i'$ are two factor sets for $E(Q_i)$ satisfying $f_i (w,w)=f_i' (w,w)=Q_i(w)$ for all $w\in W_i$, then $f_i+f'_i= \delta(t)$ is a boundary in the bar resolution. When we apply this to the extension associated to the quadratic form $Q_2 f_W=f_V Q_1: W_1 \to V_2$, we see that there is a function $t: W_1 \to V_2$ such that
\begin{equation}
\label{eqn:grouphomidentity}
(\delta t) (w,w')=t(w')+t(w+w')+t(w)= f_2 (f_W (w),f_W(w'))+f_V (f_1(w,w'))\end{equation}
for all $w,w'\in W_1$. We define $f_G : G(Q_1)\to G(Q_2)$ by $f_G (v,w)=(f_V(v)+t(w) , f_W (w))$ for all $v\in V_1$ and $w\in W_1$.
To check that $f_G$ is a group homomorphism, we need to show that $$f_G \bigl ( (v,w)(v',w') \bigl )=f_G (v,w)f_G (v',w')$$
holds for all $v,v'\in V_1$ and $w,w'\in W_1$. Writing this out in detail, one sees that this equation is equivalent to Equation \ref{eqn:grouphomidentity}, hence it holds.
So, we obtain a group homomorphism $f_G :G(Q_1) \to G(Q_2)$ as desired.  We conclude the following:

\begin{pro}
\label{pro:equivalence} The categories $\bf Quad$ and $\bf Ext$ are
equivalent.
\end{pro}

An immediate consequence of this equivalence is the following:

\begin{cor}
\label{cor:pullbacks} Let $Q_1$ and $Q_2$ be quadratic maps with
extension classes $q_1 \in H^2(W_1, V_1)$ and $q_2 \in H^2(W_2,
V_2)$ respectively. If  $f:Q_1 \to Q_2$ is a morphism of quadratic
maps, then
$$(f_W) ^* (q_2)=(f_V)_* (q_1)$$
in $H^2(W_1, V_2)$.
\end{cor}
\begin{proof} Let $Q': W_1 \to V_2$ be the quadratic map defined
by $Q'=f_V Q_1=Q_2f_W$. Then, we have
$$
\begin{CD}
W_1 @>=>> W_1 @>f_W>> W_2 \\
@VVQ_1V @VVQ'V @VVQ_2V \\
V_1 @>f_V>> V_2 @>=>> V_2 \\
\end{CD}
$$
So, we find a factorization of $f$ in $\bf Quad$ as  $Q_1
\maprt{f_1} Q' \maprt{f_2} Q_2 $. By Proposition
\ref{pro:equivalence}, we obtain a factorization of the
corresponding morphism in $\bf Ext$, this gives the following
commuting diagram:
$$
\begin{CD}
E_1: @. \ \ \  0 @>>> V_1 @>>> G(Q_1) @>>> W_1 @>>> 0 \\
@. @. @VVf_{V}V @VV(f_1)_GV @| @. \\
E': @. \ \ \  0 @>>> V_2 @>>> \widetilde G @>>> W_1 @>>> 0 \\
@. @. @| @VV(f_2)_GV @VVf_WV @. \\
E_2: @. \ \ \  0 @>>> V_2 @>>> G(Q_2) @>>> W_2 @>>> 0 \\
\end{CD}
$$
Hence, we have $(f_W) ^* (q_2)=(f_V)_* (q_1)$ as desired.
\end{proof}

\subsection{Extensions and representations of quadratic maps}
\label{subsect:extensions}

We now introduce certain categorical notions for maps between
quadratic maps such as kernel and cokernel of a map and then give
the definition of extensions of quadratic maps.

The kernel of a morphism $f: Q_1 \to Q_2$ is defined as the
quadratic map $Q_1|_{\ker f_W} : \ker f_W \to \ker f_V .$ We denote
this quadratic map as $\ker f$. If $\ker f$ is the zero quadratic
map, i.e., the quadratic map from a zero vector space to zero vector
space, then we say $f$ is injective. Similarly, we define the image
of a quadratic map $f: Q_1 \to Q_2 $ as the quadratic map $Q_2
|_{\Img f_W} : \Img f_W \to \Img f_V.$ We denote this quadratic map
by $\Img f$ and say $f$ is surjective if $\Img f=Q_2$. Given an
injective map $f : Q_1 \to Q_2 $, we say $f$ is a {\it normal
embedding} if $$B_2 (f_W(w_1), w_2 ) \in \Img f_V$$ for all $w_1 \in
W_1 $ and $w_2 \in W_2 $. Given a normal embedding $f : Q_1 \to Q_2
$, we can define the cokernel of $f$ as the quadratic map $\coker f
: \coker f_W \to \coker f_V $ by the formula $$(\coker f)( w_2 +
\Img f_W ) = Q_2 (w_2 ) + \Img f_V .$$ We are now ready to define an
extension of two quadratic maps.

\begin{defn} We say a sequence of quadratic maps of the form
\begin{equation}
\label{eqn:extQuad} {\cal E} : 0 \to Q_1 \maprt{f} Q_2 \maprt{g} Q_3
\to 0
\end{equation}
is an extension of quadratic maps if $f$ is injective, $g$ is
surjective, and $\Img f=\ker g$.
\end{defn}

Note that in an extension ${\cal E}$ as above, the first map $f:
Q_1\to Q_2$ is a normal embedding because we have $$B_2({\rm Im} f_W, W_2)=B_2(\ker g_W,
W_2) \subseteq \ker g_V= {\rm Im} f_V.$$ We say the extension ${\cal
E}$ is a split extension if there is a morphism of quadratic maps $s
: Q_3 \to Q_2 $ such that $g \circ s = \id _{Q_3}$. In this case we
write $Q_3 \iso Q_1 \rtimes Q_2$.

Later in this paper we consider the extensions where $Q_1$ is just
the identity map  $\id_U: U \to U$ of a vector space $U$. In this
case, we denote the extension by $${\cal E} : 0\to U \maprt{i}
\widetilde Q \maprt{\pi} Q \to 0,$$ and say $\cal E$ is an extension
of $Q$ with an abelian kernel $U$. In Section \ref{sect:cohdefn}, we
define obstructions for splitting such extensions and also give a
classification theorem for such extensions in a subcategory of ${\bf
Quad}$ where all the quadratic maps are assumed to be Bockstein
closed.

Given an extension of quadratic map $Q$ with an abelian kernel $U$,
there is an action of $Q$ on $U$ induced from the bilinear form
$\widetilde B$ associated to $\widetilde Q$. This action is defined
as a homomorphism
$$ \rho_W : W \to \Hom (U , U)$$
which satisfies $i_V \bigl (\rho_W (w) (u) \bigr )= \widetilde B(i_W(u), \overline w) $ where $\overline w$ is vector in $\widetilde
W$ such that $\pi_W (\overline w)=w$.

In the definition of the representation of a quadratic map, we need
the following family of quadratic maps: Let $U$ be a vector space.
We define
$$ Q_{\gl (U)} : \End(U) \to \End(U)$$
to be the quadratic map such that $ Q_{\gl(U)} (A)=A^2 +A$  for all
$A\in \End (U)$. (See Example 2.5 in \cite{PaYa}.)

\begin{defn}
A \emph{representation} of a quadratic form $Q$ is defined as a
morphism
$$\rho : Q \to Q_{\gl(U)}$$
in the category of quadratic maps. In other words, a representation
is a pair of maps $\rho=(\rho _W,\rho_V )$ such that the following
diagram commutes
$$
\begin{CD}
W @>\rho_W>> \End (U) \\
@VVQV @VVQ_{\gl(U)}V \\
V @>\rho_V>> \End(U). \\
\end{CD}
$$
\end{defn}
If $U$ is a $k$-dimensional vector space, then we say $\rho$ is a
$k$-dimensional representation of $Q$. Given a representation as
above, we sometimes say $U$ is a $Q$-module to express the fact that
there is an action of $Q$ on $\id _U: U \to U$ via the
representation $\rho$.

\section{Bockstein closed quadratic maps}
\label{sect:Bclosedmaps}

Let $E(Q)$ be a central extension of the form $0 \to V \to G(Q) \to
W\to 0$ associated to a quadratic map $Q:W \to V$, where $V$ and $W$
are $\FF _2$-vector spaces. Let $q \in H^2 (W, V)$ denote the
extension class of $E$. Choosing a basis $\{v_1,\dots , v_n \} $ for
$V$, we can write $q$ as a tuple $$q=(q_1,\dots, q_n)$$ where $q_i
\in H^2 (W, \FF _2 )$ for all $i$. The elements $\{ q_i \}$ generate
an ideal $I(Q)$ in the cohomology algebra $H^* (W, \FF _2)$. It is
easy to see that the ideal $I(Q)$ is independent of the basis chosen
for $V$, and hence is completely determined by $Q$.

\begin{defn}
\label{def:Bockstein closed} We say $Q:W \to V$ is {\it Bockstein
closed} if $I(Q)$ is invariant under the Bockstein operator on
$H^*(W;\FF _2)$. A central extension $E(Q): 0 \to V \to G(Q) \to W
\to 0$ is called Bockstein closed if the associated quadratic map
$Q$ is Bockstein closed.
\end{defn}

The following was proven in \cite{PaYa} as Proposition 3.3.

\begin{pro}
\label{pro:def of L} Let $Q: W \to V$ be a quadratic map, and let $q
\in H^2(W, V)$ be the corresponding extension class. Then, $Q$ is
Bockstein closed if and only if there is a cohomology class $L \in
H^1 (W, \End (V))$ such that $\beta (q)=Lq$.
\end{pro}

We often choose a basis for $W$ and $V$ ($\dim W=m$ and $\dim V=n$)
and express the formula $\beta (q)=Lq$ as a matrix equation. From
now on, let us assume $W$ and $V$ have some fixed basis and let $\{
x_1,\dots, x_m\}$ be the dual basis for $W$. Then, each component
$q_k$ is a quadratic polynomial in variables $ x_i $ and $L$ is an
$n\times n$ matrix with entries given by linear polynomials in
$x_i$'s. If we express $q$ as a column matrix whose $i$-th entry is
$q_i$, then $\beta(q)=Lq$ makes sense as a matrix formula where $Lq$
denotes the matrix multiplication. In general, we can have different
matrices, say $L_1$ and $L_2$, such that $\beta (q)=L_1 q =L_2 q$.
It is known that $L$ is unique when $E$ is $2$-power exact. (See
Proposition 8.1 in \cite{PaYa}.)

\begin{ex}\label{ex:gl induced}
Let $G$ be the kernel of the mod $2$ reduction map $GL_n (\ZZ /8)
\to GL _n (\ZZ /2 )$. It is easy to see that $G$ fits into a central
short exact sequence
$$0 \to \gl _n (\FF _2) \to G \to \gl _n (\FF _2) \to 0 $$
with associated quadratic map $Q_{\gl_n}$ where  $\gl_n (\FF_2)$ is the vector space of $n\times n$ matrices with entries in $\FF_2$ and the quadratic map $Q_{\gl _n}: \gl_n (\FF _2 ) \to \gl_n (\FF _2)$ is defined by $Q_{\gl _n } (\mathbb{A} )=\mathbb{A}^2 +\mathbb{A}$ (see \cite[Ex. 2.5 and 2.6]{PaYa} for more details).
We showed in \cite[Cor. 3.9]{PaYa} that this extension and
its restrictions to suitable subspaces such as $\mathfrak{sl} _n
(\FF _2)$ or $\mathfrak{u} _n (\FF _2)$ are Bockstein closed. Here
$\mathfrak{sl} _n (\FF _2 )$ denotes the subspace of $\gl_n(\FF_2)$
formed by matrices of trace zero and $\mathfrak{u} _n (\FF _2 )$
denotes the subspace of strictly upper triangular matrices. Note
that the extension for $\mathfrak{u} _n (\FF _2 )$ is also a
$2$-power exact extension (see \cite[Ex. 9.6]{PaYa}).
\end{ex}

We now consider the question of when an extension of two Bockstein
closed quadratic maps is also Bockstein closed. The equations that
we find in the process of answering this question  will give us the
motivation for the definition of the cohomology of Bockstein closed
quadratic maps.

Let $${\cal E}: 0 \to U \maprt{i} \widetilde{Q} \maprt{\pi} Q \to 0
$$ be an extension of the quadratic map $Q$ with abelian kernel $U$.
We can express this as a diagram of quadratic maps as follows:
$$
\begin{CD}
U @>i_W>> U\oplus W @>\pi_W>> W \\
@VV{\id}_UV @VV\widetilde QV @VVQV \\
U @>i_V>> U\oplus V @>\pi_V>> \ V. \\
\end{CD}
$$
Note that we have $\widetilde{Q} (u,0)=(u,0)$ and $\pi_V
\widetilde{Q} (0,w)=Q(w)$ for all $u \in U$ and $w\in W$. Also,
there is an action of $Q$ on $U$ given by the linear map $\rho _W: W
\to \End(U)$ defined by the equation
$$i_V( \rho_W(w) u)= \widetilde{B} ((u,0),(0,w)).$$
Hence, we can write
$$ \widetilde Q (u,w)=(u+\rho_W(w)u+f(w), Q(w))$$
where $f : W \to U$ is a quadratic map called factor set. We will
study the conditions on $f$ and $\rho_W$ which make $\widetilde Q$ a
Bockstein closed quadratic map.

Let $k=\dim U$. Choose a basis for $U$, and let $\{z_1,\dots ,
z_k\}$ be the associated dual basis for $U^*$. Then, we can express
the extension class $\tilde q$ of $\widetilde Q$ as a column matrix
$$\tilde q = \left[\begin{matrix} q \cr \beta(z)+Rz+f \cr
\end{matrix}\right]
$$
where $f$ and $q$ denote the column matrices for the quadratic maps
$f: W \to U$ and $Q: W \to V$ respectively. Here $z$ is the column
matrix with $i$-th entry equal to $z_i$ and $R$ is a $k\times k$
matrix with entries in $x_i$'s which is associated to $\rho _W : W
\to \End(U)$. Applying the Bockstein operator, we get
$$\beta(\tilde q) = \left[\begin{matrix} \beta(q) \cr \beta(R)z+
R\beta(z)+\beta(f) \cr
\end{matrix}\right].$$
Note that $\widetilde Q$ is Bockstein closed if we can find
$\widetilde L$ such that $\beta (\tilde q)=\widetilde L \tilde q$.
Since $\beta (q)=Lq$ for some $L$, we can take $\widetilde L$ as
$$\widetilde L = \left[\begin{matrix} L & 0  \cr L_{2,1} & L_{2,2} \cr
\end{matrix}\right].
$$
Note that assuming the top part of the matrix $\widetilde L$ is in a
special form does not affect the generality of the lower part. So,
under the assumption that $Q$ is Bockstein closed, the quadratic map
$\widetilde Q$ is Bockstein closed if only if there exist $L_{2,1}$
and $L_{2,2}$ satisfying
\begin{equation}\label{eqn:requirement}
\beta(R)z+R\beta(z)+\beta(f)= L_{2,1}q+L_{2,2} (\beta (z)+Rz+f).
\end{equation}
We have $$L_{2,2}= \sum _{i=1} ^k L_{2,2} ^{(i)} z_i + \sum _{j=1}
^m L_{2,2} ^{(j)}x_j$$ where $L_{2,2}^{(i)}$ and $L_{2,2}^{(j)}$ are
scalar matrices, so we can write $L_{2,2}=L_{2,2}^z+L_{2,2} ^x$
where $L_{2,2}^z$ is the first sum and $L_{2,2}^x $ is the second
sum in the above formula.

Equation $(\ref{eqn:requirement})$ gives $L_{2,2}^z \beta(z)=0$
which implies $L_{2,2}^z=0$. Writing $L_{2,2}=L_{2,2}^x$ in
$(\ref{eqn:requirement})$, we easily see that we must have
$L_{2,2}=R$. Putting this into $( \ref{eqn:requirement})$, we get
\begin{equation}
\label{eqn:obstruction} [ \beta(R)+R^2]z+[\beta (f)+Rf]=L_{2,1}q.
\end{equation}

As we did for $L_{2,2}$, we can write $L_{2,1}$ also as a sum
$L_{2,1}=L_{2,1}^z +L_{2,1}^x$ where the entries of $L_{2,1}^z$ are
linear polynomials in $z_i$'s and the entries of $L_{2,1}^x$ are
linear polynomials in $x_i$'s. So, Equation \ref{eqn:obstruction}
gives two equations:
\begin{equation}
\label{eqn:twoequations}
\begin{split}
[\beta(R)+R^2 ]z&=L_{2,1}^z q \\
\beta (f)+Rf&=L_{2,1}^x q. \\
\end{split}
\end{equation}

From now on, let us write $Z=L_{2,1}^z$. Note that $Z$ is a $k\times
n$ matrix ($k=\dim U$ and $n=\dim V$) with entries in the dual space
$U^*$, so it can be thought of as a linear operator $Z: U \to
\Hom(V,U).$ Viewing this as a bilinear map $U \times V \to U$, and
then using an adjoint trick, we obtain a linear map $\rho_V: V \to
\text{Hom}(U,U)=\text{End}(U).$ As a matrix, let us denote $\rho_V$
by $T$. The relation between $Z$ and $T$ can be explained as
follows: If $\{ u_1,\dots, u_k\}$ are basis elements for $U$ dual to
the basis elements $\{z_1,\dots, z_k\}$ of $U^*$ and if  $\{
v_1,\dots, v_k\}$ is the  basis for $V$ dual to the basis elements
$\{t_1,\dots, t_k\}$ of $V^*$, then $Z(u_i)(v_j)=T(v_j)(u_i)$ for
all $i,j$. So, if $Z=\sum _{i=1} ^k Z(i) z_i$ and $T=\sum _{j=1}^n
T(j)t_j$, then we have $Z(i)e_j=T(j)e_i$ where $e_i$ and $e_j$ are
$i$-th and $j$-th unit column matrices. This implies, in particular,
that
$$ Zq= T(q)z$$
where $T(q)$ is the matrix obtained from $T$ by replacing $t_i$'s
with $q_i$'s. So, the first equation in (\ref{eqn:twoequations}) can
be interpreted as follows:

\begin{lem}\label{lem:repeqn}  Let $\rho _W: W \to \End(U) $ and
$\rho _V : V \to \End (U)$ be two linear maps with corresponding
matrices $R$ and $T$. Let $Z$ denote the matrix for the adjoint of
$\rho_V$ in  $\Hom (U, \Hom (V,U))$. Then, the equation
$$[\beta(R)+R^2]z=Zq$$ holds if and only if $\rho= (\rho_W , \rho
_V):Q \to Q_{\mathfrak{gl}(U)}$ is a representation.
\end{lem}

\begin{proof} Note that the diagram
$$
\begin{CD}
W @>\rho_W>> \End (U) \\
@VVQV @VVQ_{\gl(U)}V \\
V @>\rho_V>> \End(U) \\
\end{CD}
$$
commutes if and only if $$\beta(R)+R^2= T(q)$$ where $T(q)$ is the
$k\times k$ matrix obtained from $T$ by replacing $t_i$'s with
$q_i$'s. We showed above that $Zq=T(q)z$, so $\beta(R)+R^2= T(q)$
holds if and only if
$$[\beta(R)+R^2]z= T(q)z=Zq.$$ This completes the proof.
\end{proof}

As a consequence of the above lemma, we can conclude that if the
action of $W$ on $\End(U)$ comes from a representation $\rho : Q \to
Q_{ \gl(U)}$, then the only obstruction for a quadratic map
$\widetilde Q $ to be Bockstein closed is the second equation $$
\beta(f)+Rf=L_{2,1}^x q$$ given in (\ref{eqn:twoequations}). Note
that both $R$ and $L_{2,1}^x$ are matrices with entries in $x_i$'s,
so this equation can be interpreted as saying that $\beta(f)+Rf= 0$
in $A^*(Q)=\FF _2 [x_1, \dots , x_m]/(q_1,\dots , q_n )$. In the next section, we define the cohomology of a
Bockstein closed quadratic map using this interpretation.

\section{Cohomology of Bockstein closed quadratic maps}
\label{sect:cohdefn}

Let $Q : W \to V$ be a Bockstein closed quadratic map where $W$ and
$V$ are $\FF _2$-vector spaces of dimensions $m$ and $n$,
respectively. Let $U$ be a $k$-dimensional $Q$-module with
associated representation $\rho : Q \to Q_{\gl (U)}$. We will define
the cohomology of $Q$ with coefficients in $U$ as the cohomology of
a cochain complex $C^* (Q, U)$. We now describe this cochain
complex.

Let $A(Q)^*$ denote the $\FF _2$-algebra
$$A^* (Q)=\FF _2[x_1, \dots , x_m ]/(q_1,\dots, q_n)$$
as before. The algebra $A^*(Q)$ is a graded algebra where the
grading comes from the usual grading of the polynomial algebra. We
define $p$-cochains of $Q$ with coefficients in $U$ as
$$C^p(Q, U)=A(Q)^p\otimes U.$$
We describe the differentials using a matrix formula. Choosing a
basis for $U$, we can express a $p$-cochain $f$ as a $k \times 1$
column matrix with entries $f_i \in A^p (Q)$. Let $R \in H^1(W,
\End(U))$ be the cohomology class associated to $\rho_W$. We can
express $R$ as a $k\times k$ matrix with entries in $x_i$'s. We
define the boundary maps
$$ \delta  : C^p (Q, U) \to C^{p+1} (Q, U)$$
by $$ \delta (f)= \beta(f)+Rf.$$ Note that
$$\delta ^2 (f) =
\beta(R)f+R\beta(f)+R\beta(f)+R^2f=[\beta(R)+R^2]f=0$$ in $A^*(Q)$
because $\beta(R)+R^2=T(q)$ by the argument given in the proof of
Lemma \ref{lem:repeqn}. So, $C^* (Q, U)$ with the above boundary
maps is a cochain complex.

Note that although the definition of $\delta$ only uses $R$, i.e.,
$\rho_W$, the existence of $\rho_V$ is needed to ensure that $\delta
^2= 0$. Thus both maps in the structure of $U$ as a $Q$-module play
a role in establishing $\delta$  as a differential. Also note that
we need the quadratic map $Q$ to be Bockstein closed for the
well-definedness of the differential $\delta$.

\begin{defn} The cohomology of a Bockstein closed quadratic form
$Q$ with coefficients in a $Q$-module $U$ is defined as $$H^* (Q,
U):= H^* (C^* (Q, U), \delta )$$ where $C^* (Q,U)=A^*(Q)\otimes U$
and the boundary maps $\delta $ are given by $\delta
(f)=\beta(f)+Rf$.
\end{defn}

Let $U$ be the one dimensional trivial $Q$-module, i.e.,
$\rho=(\rho_W, \rho_V)=(0,0)$. In this case we write $U=\FF _2$.
Then, $H^*(Q, \FF _2 )$ is just the cohomology of the complex
$A(Q)^*$ and the boundary map $\delta$ is equal to the Bockstein
operator. In this case, the cohomology group $H^* (Q, \FF _2)$ also
has a ring structure coming from the usual multiplication of
polynomials. Note that given two cocycles $f,g \in A^*(Q)$, we have
$\beta (fg)=\beta(f)g+f\beta (g)=0 $ modulo $I(Q)$. So, we define
the product of two cohomology classes $[f],[g]\in H^* (Q, \FF _2 )$
by $$[f][g]=[fg]$$ where $fg$ denotes the usual multiplication of
polynomials.

Given a morphism $\varphi: Q_1 \to Q_2$, we have $(\varphi_W)^*
(q_2)=(\varphi_V)_* (q_1)$ by Corollary \ref{cor:pullbacks}. This
shows that $ (\varphi_W)^*: H^* (W_2, \FF _2 )\to H^* (W_1 , \FF _2
)$ takes the entries of $q_2$ into the ideal $I(Q_1)$. Thus,
$\varphi_W$ induces an algebra map $\varphi^*: A^* (Q_2)\to A^*
(Q_1)$ which gives a chain map $C^* (Q_2, U_2) \to C^* (Q_1, U_1)$
where $U_2$ is representation of $Q_2$ and $U_1$ is a representation
of $Q_1$ induced by $\varphi$. So, $\varphi$ induces a homomorphism
$$\varphi ^* : H^* (Q_2 , U_2 ) \to H^* (Q_1, U_1 ).$$ If
$U_1=U_2=\FF _2$, the induced map is also an algebra map.

In the rest of the section, we discuss the interpretations of low
dimensional cohomo-logy, $H^i(Q, U)$ for $i=0,1,2$, in terms of
extension theory. First we calculate $H^0 (Q, U)$. Note that $C^0(Q,
U)=U$ and given $u \in C^0(Q, U)$, we have $\delta (u)
(w)=\rho_W(w)u$. So, $H^0(Q, U)=U^Q$ where
$$U^Q=\{ u \mid \rho _W (w) (u)=0 \ {\rm for \ all} \ w\in W \}.$$
Note that this is analogous to Lie algebra invariants $$U^{\g}=\{ u
\mid x \cdot u =0 \ {\rm for \ all} \ x\in \g \}.$$ We refer to the
elements of $U^Q$ as $Q$-invariants of $U$.

Now, we consider $H^1(Q, U)$. Note that $C^1 (Q, U)\cong \Hom (W,
U)$. Let $d_W : W \to U$ be a $1$-cochain. By our definition of
differentials, $d_W$ is a derivation if and only if $\delta
(d_W)=\beta(d_W)+Rd_W=0$ in $A(Q)^2 \otimes U$. The last equation
can be interpreted as follows: There is a linear map $d_V : V \to U$
such that $$\bigl (1+\rho_W(w)\bigr ) d_W(w)+d_V (Q(w))=0.$$ A
trivial derivation will be a derivation $d_W: W \to U$ of the form
$d_W(w)=\rho_W(w)u$ for some $u \in U$. Note that when $U$ is a
trivial module, $d _W: W \to U$ is a derivation if and only if there
is a linear map $d_V : V \to U$ such that the following diagram
commutes
$$
\begin{CD}
W @>d_W>> U \\
@VVQV @VV{\id}V \\
V @>d_V>> \ U\, .\\
\end{CD}
$$
So, when $U$ is a trivial $Q$-module, we have $$H^1(Q, U)\iso
\Hom_{\bf Quad} (Q, U).$$ If $U=\FF _2 $, then
$$H^1 (Q, \FF _2 )=\ker \{ \beta : H^1 (W, \FF _2 ) \to A^2(Q)\}.$$
So, we have the following:

\begin{pro}
\label{pro:trivialcoef} Let $Z(Q)$ be the vector space generated by
$k$-invariants $q_1,\dots, q_n$ and let $Z(Q)^{\beta}= \{ z \in Z(Q)
\ | \ \beta (z)=0 \}$. Then, $H^1(Q, \FF _2 ) \iso Z(Q)^{\beta }$.
\end{pro}

We refer to the elements of $Z(Q)^{\beta}$ as the Bockstein
invariants of $Q$. For an arbitrary $Q$ module $U$, we have the
following:

\begin{pro}
\label{pro:H^1} There is a one-to-one correspondence between $H^1(Q,
U)$ and the splittings of the split extension $0 \to U \to U \rtimes
Q \to Q \to 0$.
\end{pro}

\begin{proof}
Observe that $s: Q \to U \rtimes Q$ is a morphism in $\bf Quad$ if
and only if $ \widetilde Q (s_W (w)) = s_V (Q (w)) $. Since $\pi
s=\id$, we can write $s_W(w)=(d_W(w), w)$ and $s_V (v)=(d_V (v),
v)$. So, $s$ is a morphism in $\bf Quad$ if and only if $
d_W(w)+\rho(w)d_W(w)=d_V (Q(w))$, i.e., $d_W$ is a derivation. It is
easy to see that trivial derivation corresponds to a splitting which
is trivial up to an automorphism of $U\rtimes Q$.
\end{proof}

We now consider extensions of a Bockstein closed quadratic map $Q$
with an abelian kernel $U$, and show that $H^2(Q, U)$ classifies
such extensions up to an equivalence. If there is a diagram of
quadratic maps of the following form
\begin{equation}
\label{eqn:equivalence}
\begin{CD}
{\cal E}_1  : \ \ @. 0 @>>> U @>>> Q_1 @>>> Q @>>> 0 \\
@. @. @| @VV{\varphi}V @| @. \\
{\cal E}_2  : \ \ @. 0 @>>> U @>>> Q_2 @>>> Q @>>> \ 0 \, , \\
\end{CD}
\end{equation}
then we say ${\cal E}_1$ is equivalent to ${\cal E}_2$. Let ${\rm
Ext}(Q, U)$ denote the set of equivalence classes of extensions of
the form $0 \to U \to \widetilde Q \to Q \to 0$ with abelian kernel
$U$ where $\widetilde Q$ and $Q$ are Bockstein closed. We can define
the summation of two extensions as it is done in  group extension
theory. So, ${\rm Ext} (Q, U)$ is an abelian group. We prove the
following:
\begin{pro}
\label{pro:H^2} $H^2(Q, U) \iso {\rm Ext}(Q, U)$.
\end{pro}
\begin{proof}
Note that we already have a fixed decomposition for the domain and
the range of $\widetilde Q$, so we will write our proof using these
fixed decompositions. We skip some of the details which are done
exactly as in the case of group extensions.

First we show there is a 1-1 correspondence between 2-cocycles and
Bockstein closed extensions. Recall that a 2-cocycle is a quadratic
map $f: W \to U$ such that $\beta(f)+Rf =0$ in $A(Q)^*$. Consider
the extension ${\cal E} : 0 \to U \to \widetilde Q \to Q \to 0$
where $$\widetilde Q(u,w)=(u+\rho_W(w)u+f(w), Q(w)).$$ We have seen
earlier that $\widetilde Q$ is Bockstein closed if and only if
$\beta(f)+Rf=0$ in $A(Q)^*$. So, $\cal E$ is an extension of
Bockstein closed quadratic maps if and only if $f$ is a cocycle.

Now, assume that ${\cal E}_1$ and ${\cal E}_2$ are two equivalent
extensions. Let $\varphi : Q_1 \to Q_2 $ be a morphism which makes
the diagram (\ref{eqn:equivalence}) commute. Then, we can write
$\varphi _W (u,w)=(u+a(w), w)$ and $\varphi _V (u,v)=(u+b(v), v)$.
Let $f_1$ and $f_2$ be the cocycles corresponding to extensions
${\cal E} _1 $ and ${\cal E}_2$ respectively. Then, the identity
$Q_2 (\varphi _W (u,w)) =\varphi _V (Q_1 (u,w))$ gives
$$ \bigl (u+a(w)+\rho(w)( u+a(w)) +f_2 (w),\ Q(w)\bigr )=
\bigl (u+\rho(w)u+f_1 (w)+b(Q(w)),\ Q(w)\bigr ).$$ So, we have
\begin{equation}\label{eqn:difference}
f_2(w)+f_1(w)=\bigl ( 1+\rho(w)\bigr )a(w)+b(Q(w)).
\end{equation}
Thus $f_1+f_2=\delta (a) $ in $A(Q)^*$. Conversely, if
$f_1+f_2=\delta (a)$ in $A(Q)^*$, then there is a $b : V \to U$ such
that equation (\ref{eqn:difference}) holds, so we can define the
morphism $\varphi : Q_1 \to Q_2 $ as above so that the diagram
(\ref{eqn:equivalence}) commutes.
\end{proof}

\section{LHS-spectral sequence for 2-power
exact extensions} \label{sect:mod2cohomology}

In this section we study the Lyndon-Hochschild-Serre (LHS) spectral
sequence associated to a Bockstein closed $2$-power exact extension.
We first recall the definition of a $2$-power exact extension.
\begin{defn}\label{defn:2-powerexact}
A central extension of the form $$E(Q): 0 \to V \to G(Q) \to W \to
0$$ with corresponding quadratic map $Q: W \to V$ is called {\it
$2$-power exact} if the following conditions hold:
\begin{enumerate}
\item[(i)] $\dim(V)=\dim(W)$,
\item[(ii)] the extension is a Frattini extension, i.e., image of $Q$
generates $V$, and
\item[(iii)] the extension is effective, i.e., $Q(w)=0$ if and only if
$w=0$.
\end{enumerate}
\end{defn}

In this section, we calculate the mod-$2$ cohomology of $G(Q)$ using
a LHS-spectral sequence when $E(Q)$ is a Bockstein closed $2$-power
exact extension. The mod-$2$ cohomology ring structure of $2$-power
exact groups has a simple form and it is not very difficult to
obtain once certain algebraic lemmas are established. Similar
calculations were given by Rusin \cite[Lemma 8]{Rusin} and
Minh-Symonds \cite{MiSy}.

We first prove an important structure theorem concerning the
$k$-invariants of Bockstein closed 2-power exact extensions.

\begin{pro}
\label{pro:regularsequence}  Let $E(Q): 0 \to V \to G(Q) \to W \to
0$ be a Bockstein closed, 2-power exact extension with $\dim (W)=n$.
Then, the $k$-invariants $q_1,\dots,q_n$, with respect to some basis
of $V$, form a regular sequence in $H^*(W, \FF_2)=\FF _2[x_1,\dots,
x_n]$ and $A^*(Q)=\FF _2 [x_1,\dots, x_n]/(q_1,\dots, q_n)$ is a
finite dimensional $\FF _2$-vector space.
\end{pro}

\begin{proof}
We have shown in \cite[Proposition 7.8]{PaYa} that the
$k$-invariants $q_1,\dots,q_m$ form a regular sequence in $H^*(W;\FF
_2)$. This sequence is regular in any order. To show the second
statement, let $K$ denote the algebraic closure of $\FF_2$. Since
the dimension of the variety associated to $I(Q)=(q_1,\dots,q_m)$ is
zero, the (projective) Nullstellensatz shows that $A^*(Q)$ is a
nilpotent algebra (elements $u$ in positive degree have $u^k=0$ for
some $k$, depending on $u$). However since $A^*(Q)$ is a finitely
generated and commutative algebra, this shows that $A^*(Q)$ is
finite dimensional as a vector space over $K$.
\end{proof}

Recall that a regular sequence in a polynomial algebra is always
algebraically independent (see \cite[Prop. 6.2.1]{Smith3}). So, if
$E(Q)$ is a Bockstein closed $2$-power extension, then the
subalgebra generated by the $k$-invariants $q_1,\dots, q_n$ is a
polynomial algebra.  We denote this subalgebra by $\FF _2[q_1,\dots,
q_n]$. We have the following:

\begin{pro}\label{pro:freeover} Let $E(Q): 0 \to V \to G(Q) \to W \to
0$ be a Bockstein closed, 2-power exact extension with $\dim (W)=n$.
Then, $H^*(W;\FF _2)$ is free as a $\FF _2 [q_1,\dots,q_n]$-module.
As \ $\FF _2[q_1,\dots,q_n]$-modules $H^*(W;\FF_2) \cong
\FF_2[q_1,\dots,q_n] \otimes A^*(Q)$  where $A^*(Q)$ is given the
trivial module structure.
\end{pro}

\begin{proof} Let $P=\FF_2[q_1,\dots, q_n]$ be the subalgebra generated
by $q_1,\dots,q_n$ in $H^*(W;\FF _2)$. Since $A^*(Q)$ is finite
dimensional, $H^*(W;\FF_2)$ is finitely generated over $P$. For
example, if we take $\hat{A}^*(Q)$ a $\FF_2$-vector subspace of
$H^*(W;\FF_2)$ mapping $\FF_2$-isomorphically to  $A^*(Q)$ under the
projection $H^*(W;\FF_2) \to A^*(Q)$, then a homogeneous
$\FF_2$-basis of $\hat{A}^*(Q)$ gives a set of generators of
$H^*(W;\FF_2)$ as a $P$-module.

To prove this, one inducts on the degree of $\alpha \in
H^*(W;\FF_2)$. Let $\{ a_i \mid i \in I \}$ be a $\FF_2$-basis of
$\hat{A}^*(Q)$. We want to show $\alpha$ is in the $P$-span of the
$a_i$. For degree of $\alpha$ equal to one or two, this is immediate
since the degree of $q_i$ is 2 for all $i$. In general, by
subtracting a suitable linear combination of $a_i$'s from $\alpha$,
we get an element $\beta$ in the ideal $(q_1,\dots,q_n)$. To show
$\alpha$ is in the $P$-span of the $a_i$, it is enough to show
$\beta$ is. However $\beta = \sum_{i=1}^n \beta_i q_i$ where the
degree of the $\beta_i$ is 2 less than the degree of $\beta$. By
induction, each $\beta_i$ and hence $\beta$ is in the $P$-span of
the $a_i$ and so we are done.

Since $H^*(W;\FF_2)$ is a polynomial algebra, it is trivially
Cohen-Macaulay. Since  $H^*(W;\FF _2)$ is a finitely generated
$P$-module, the fact that $H^*(W;\FF_2)$ is a free $P$-module
follows from the fact that $P$ is a polynomial algebra (see Theorem
5.4.10 of \cite{BensonII} for example). Thus $H^*(W;\FF _2)$ is a
free $P$-module.

Finally if $\{ b_j \mid j \in J \}$ is a basis for the free
$P$-module $H^*(W;\FF_2)$, then every element $\alpha \in
H^*(W;\FF_2)$ can be written uniquely in the form $\sum_{j \in J}
\alpha_j b_j$ where $\alpha_j \in \FF _2[q_1,\dots,q_n]$. Projecting
to $A^*(Q)$, this shows every element of $A^*(Q)$ can be written
uniquely as a span of the corresponding images of the $b_j$. In
other words, the $\{ b_j \mid j \in J \}$ projects to a basis of
$A^*(Q)$. Thus one can define a map of $\FF
_2[q_1,\dots,q_n]$-modules
$$\FF_2[q_1,\dots,q_n] \otimes A^*(Q) \to H^*(W;\FF_2)$$ which is an
isomorphism.
\end{proof}

Now, consider the LHS-spectral sequence in mod-$2$ coefficients
$$E_2^{p,q} = H^p (W, H^q (V, \FF _2 )) \Rightarrow H^{p+q} (G(Q), \FF_2)$$
associated to the extension $E(Q)$. Let $\{x_1,\dots , x_n \}$
denote a basis for the dual of $W$ and let $\{ t_1,\dots , t_n \}$
be a basis for the dual of $V$.

\begin{lem}
\label{lem:image of d_2} For each $1\leq k \leq m$, we have $d_2
(t_k)=q_k$.
\end{lem}

\begin{proof}
From standard theory, the central extension $E(Q)$ corresponds to a
principal $BV$-bundle $BV \to BG(Q) \to BW$ with classifying map $f:
BW \to BBV$ where $BBV=K(V,2)$. Also $H^2(BBV, V)\cong \Hom(V,V)$
and $q=f^*(\id)$ where $\id \in \Hom(V, V)$ is the identity map.

In the LHS-spectral sequence for the fibration $BV \to EBV \to
K(V,2)$, we have $d_2: H^1(BV , V) \to H^2(BBV, V)$ is an
isomorphism (since $EBV$ is contractible). In fact, it identifies
$H^1(BV, V)=H^2(BBV, V)=\Hom(V, V)$, so $d_2(\id)=\id$. This
identity pulls back to our fibration $BV \to BG(Q) \to BW$ as
$d_2(\id)=f^*(\id)=q$. From this the lemma easily follows after
taking a basis for $V$.
\end{proof}

Now we are ready for some computations.

\begin{thm}
\label{thm:mod 2 cohomology} Let $E(Q): 0 \to V \to G(Q) \to W \to
0$ be a Bockstein closed, $2$-power exact extension with $\dim
(W)=n$. Then
$$H^*(G (Q);\FF_2) \cong \FF_2[s_1,\dots,s_n] \otimes A^*(Q)$$
as graded algebras, where $deg(s_i)=2$ for $i=1,\dots, n$.
\end{thm}

\begin{proof}
Consider the LHS-spectral sequence for the (central) extension
$E(Q)$. The $E_2$-page has the form
$$E_2^{*,*}=H^*(W , \FF_2) \otimes H^*(V, \FF _2)
=\FF_2[x_1,\dots,x_n] \otimes \FF _2 [t_1,\dots,t_n].$$ We have
previously seen that $d_2(t_k)=q_k$. Here we have chosen basis for
$W$,$V$ and their duals as done previously. Note that $d_2(t_k^2)=0$
and $d_3(t_k^2)=\beta(q_k)$ by a standard theorem of Serre. Let
$$\wedge^*(t_1,\dots,t_n)$$ be the $\FF_2$-subspace of $H^*(V,
\FF_2)$ generated by the monomials $t_1^{\epsilon_1}\dots
t_n^{\epsilon_n}$ where $\epsilon_i = 0,1$ for each $i=1,\dots,n$.
Then, $$H^*(V , \FF _2) \cong \FF _2 [t_1^2,\dots,t_k^2] \otimes
\wedge^*(t_1,\dots,t_n)$$ as $\FF_2$-vector spaces. (Not as
algebras!) Using Proposition \ref{pro:freeover}, we also write
$$H^*(W, \FF_2) \cong \FF_2 [q_1,\dots,q_n] \otimes A^*(Q)$$ as
vector spaces. Thus as a differential graded complex, $(E_2^{*,*},
d_2)$ splits as a tensor product
$$E_2^{*,*} \cong \FF _2 [t_1^2,\dots,t_n^2] \otimes A^*(Q) \otimes
\bigl ( \wedge^*(t_1,\dots,t_n) \otimes \FF _2[q_1,\dots,q_n],
d_2\bigr )$$ where the differential on the first two tensor summands
is trivial. By K\" unneth's theorem, $$E_3^{*,*} \cong \FF _2
[t_1^2,\dots,t_n^2] \otimes A^*(Q) \otimes
H^*(\wedge^*(t_1,\dots,t_n) \otimes \FF _2 [q_1,\dots,q_n], d_2).$$
Again by K\" unneth's theorem, the final term is just $H^*(pt, \FF
_2)$ since it breaks up as the tensor of
$$H^*\bigl (\wedge^*(t_j) \otimes \FF _2 [q_j],\ d_2(t_j)=q_j \bigr )$$
which is the cohomology of a point. Thus
$$E_3^{*,*} = \FF _2[t_1^2,\dots,t_n^2] \otimes A^*(Q)$$
with $d_3(t_j^2)=\beta(q_j)$. Since $Q$ is Bockstein closed,
$\beta(q_j)=0$ in $A^*(Q)$ and so $d_3(t_j^2)=0$. Thus
$E_3^{*,*}=E_4^{*,*}$. By dimensional considerations, there can be
no further differentials in the spectral sequence and so we see
$E_{\infty}^{*,*}=E_3^{*,*}$.

This says, in particular, that there are elements $s_i \in H^*(G,
\FF_2)$ such that $$\res^{G}_V (s_i)=t_i^2.$$ Since the $t_i$'s are
algebraically independent in $H^*(V, \FF_2)$, we conclude that the
$s_i$'s are algebraically independent in $H^*(G, \FF_2)$. Thus when
we define an $\FF_2$-vector space homomorphism $$\FF_2
[s_1,\dots,s_n] \otimes A^*(Q) \to H^*(G, \FF_2 )$$ using the
inclusion map on the first factor and the inflation map on the
second, we will have a well-defined map of algebras. (Note that the
inflation map is always an algebra map.) Finally by the structure of
$E_{\infty}^{*,*}$, it is clear that this map is onto, and since our
algebras have finite type, this means that it is an isomorphism of
algebras as desired.
\end{proof}

\begin{rem} Note that although we assumed that the extension $E(Q)$
is $2$-power exact in the above calculation, we only use the fact
that the $k$-invariants $q_1\dots, q_n$ form a regular sequence. If
$E(Q): 0 \to V \to G(Q) \to W \to 0$ is a Bockstein closed extension
with $\dim V=n$, $\dim W=m$ such that the $k$-invariants $q_1\dots,
q_n$ form a regular sequence, then the $k$-invariants will be
algebraically independent in $\FF_2[x_1,\dots, x_m]$, in particular,
they will be linearly independent. This shows that in this case,
$E(Q)$ is a Frattini extension, i.e., $\{ Q(w) \mid w\in W \}$
generates $V$. Also, we must have $n \leq m$ for dimension reasons.
The assumption that $Q$ is Bockstein closed will imply that the
variety of $I(Q)$ is $\FF_2$-rational, so $W$ must have a $m-n$
dimensional subspace $W'$ such that $Q$ restricted to $W'$ is zero.
But it may happen that this $W'$ has nontrivial commutators with the
rest of the elements in $W$. So, we can not conclude that $G(Q)$
splits as $G(Q)\iso G' \times \ZZ /2$. Hence, these groups are still
interesting and the calculation above shows that the mod-2
cohomology of these groups is also in the form $H^*(G, \FF_2)\iso
A^*(Q) \otimes \FF_2[s_1,...,s_n].$
\end{rem}

The group extensions associated to the strictly upper triangular
matrices $\mathfrak{u} _n (\FF _2 )$ are Bockstein closed $2$-power
exact extensions (see Example \ref{ex:gl induced}). So, we have a
complete calculation for the mod-2 cohomology of these groups. To
calculate the Bockstein's of the generators of mod-2 cohomology of a
$2$-power exact extension, we need to consider also the
Eilenberg-Moore spectral sequence associated to the extension.

\section{Eilenberg-Moore Spectral Sequence}
\label{section: EMSS}

In this section we study the Eilenberg-Moore Spectral Sequence of a
Bockstein closed 2-power exact extension
$$E(Q): 0 \to V \to G(Q) \to W \to 0$$
associated to a quadratic map $Q: W \to V$. Although we have already
found the algebra structure of the mod-$2$ cohomology of $G(Q)$ as
$$H^*(G(Q);\mathbb{F}_2) \cong \mathbb{F}_2[s_1,\dots,s_n]
\otimes A^*(Q)$$ using the LHS-spectral sequence, we find the
EM-spectral sequence more useful in studying the Steenrod algebra
structure of $H^*(G(Q);\mathbb{F}_2)$. In fact, the theorems laid
out in Larry Smith's paper \cite{Smith1} directly compute the
behavior of the EM-spectral sequence in our case and give us the
Steenrod structure we desire. In order to give a fuller picture of
the underlying topology, we summarize some essentials of the
EM-spectral sequence here.

Consider a pullback square of spaces
$$
\begin{CD}
X \times_B Y @>>> Y \\
@VVV @VVpV \\
X @>f>> B \\
\end{CD}
$$
where $p: Y \to B$ is a fibration and $B$ is simply-connected. The
space $X \times_B Y$ is given as
$$X \times_B Y =\{ (x,y) \in X \times Y \mid f(x)=p(y) \}$$
and can be viewed as an amalgamation of $X$ and $Y$ over $B$. Fix
$k$ a field, and let $C^*(X)$ denote the dga (differential graded
algebra) given by the cochain complex of $X$ with coefficients in
$k$. One can show that the natural map $\alpha: C^*(X)
\otimes_{C^*(B)} C^*(Y) \to C^*(X \times_B Y)$ yields an isomorphism
$$\Tor_{C^*(B)}(C^*(X),C^*(Y)) \cong H^*( X \times_B Y)$$
of algebras. (Though one should be careful in interpreting the
algebra structure of the Tor term.) For more details on this
isomorphism see \cite[Thm 3.2]{Smith1}.

One then can show algebraically, through various filtrations of
resolutions computing the Tor term above, that there is a spectral
sequence starting at
$$E_2^{*,*}=\Tor_{H^*(B)}(H^*(X),H^*(Y))$$ converging to
$\Tor_{C^*(B)}(C^*(X),C^*(Y)) \cong H^*( X \times_B Y)$.

Explicitly, this can be constructed as a second quadrant spectral
sequence in the $(p,q)$-plane using the {\bf bar resolution}. The
picture of the $E_1$-page using the bar resolution is as follows: on
the $p=-n$ line, one has the algebra
$$H^*(X) \otimes \bar{H}^*(B) \otimes \dots \otimes \bar{H}^*(B)
\otimes H^*(Y)$$ where all tensor products are over $k$, $\bar{H^*}$
denotes the positive degree elements of $H^*$, and $n$ factors of
$\bar{H}^*(B)$ are used in the above tensor product. We place the
graded complex above on the $p=-n$ line in such a way that the
elements of total degree $q$ are placed at the $(-n, q)$ lattice
point in the $(p,q)$-plane. Let $[a|b_1|\dots|b_n|c]$ be short hand
for $a \otimes b_1 \otimes \dots \otimes b_n \otimes c$. The
differential $d_1$ is horizontal moving one step to the right and is
given explicitly in characteristic 2 (we will only use this case and
we do this also to avoid stating the signs!) by:
\begin{eqnarray*}
d_1([a|b_1| \dots |b_n|c]) &=& [af^*(b_1)|b_2|\dots |b_n|c] \\
&&+\Sigma_{i=1}^{n-1} [a|b_1| \dots |b_ib_{i+1} |\dots |b_n|c] \\
&&+[a|b_1| \dots |b_{n-1}|p^*(b_n)c]
\end{eqnarray*}
for all $a \in H^*(X), b_i \in \bar{H}^*(B), c \in H^*(Y)$.

The power of the EM-spectral sequence is the availability of this
geometric resolution to represent its $E_1$-term. For example the
$p=-1$ line can be interpreted as a portion of $H^*(X \times B
\times Y)$. If $A$ is the kernel of $d_1$ on this line, then $A$ is
the set of elements of the form $[a|b|c]$ satisfying
$$[af^*(b)|c]=[a|p^*(b)c],$$
and the elements of $A$ are permanent cycles in the EM-spectral
sequence. As shown in \cite{Smith1}, there is a Steenrod module
structure on the $p=-1$ and $p=0$ lines via the natural
identification of them inside $H^*(X \times B \times Y)$ and $H^*(X
\times Y)$ respectively. Furthermore, this Steenrod module structure
persists through all pages of the spectral sequence. Finally the
associated filtration on $H^*(X \times_B Y)$ is a filtration of
Steenrod modules and components of the associated graded module
agree with the Steenrod module structure on the $E_{\infty}$-page on
the $p=0$ and $p=-1$ lines of the EM-spectral sequence with the bar
resolution. In fact various authors have shown that the whole
EM-spectral sequence has a natural structure of a module over the
Steenrod algebra (using the bar resolution model) with ``vertical''
and ``diagonal'' Steenrod operations (corresponding to the fact that
the EM-spectral sequence is not an unstable module over ${\cal
A}_2$). We won't need these here - we will only need the structure
on the $p=0$ and $p=-1$ lines which was already laid out in
$\cite{Smith1}$ very naturally.

The bar resolution mentioned above is a nice natural resolution that
can be used to compute $E_2^{*,*}=\Tor_{H^*(B)}(H^*(X),H^*(Y))$ in
the EM-Spectral sequence and has the benefits of being ``natural''
and ``geometric'' and hence can be used to study natural operations
on the spectral sequence. However typically this resolution is too
big to carry out computations directly. In the case where $H^*(B)$
is a polynomial algebra, a smaller Koszul resolution is available to
compute this tor-term and the $E_2$-page is more tractable. Finally
if the ideal generated by $f^*(H^*(B))$ in $H^*(X)$ is nice enough,
then a change of rings isomorphism can also be used to simplify the
calculations immensely. For details see \cite{Smith1}. We will
summarize the main result below after a necessary definition.

\begin{defn}
Let $k$ be a field and $\Lambda$ be a graded commutative algebra
over $k$. An ideal $I$ in $\Lambda$ is called a {\it Borel ideal} if
there is a regular sequence $\{ x_1, x_2, \dots, x_n, \dots \}$
(either finite or infinite)  that generates the ideal $I$. Recall
that a regular sequence $\{ x_1, x_2, \dots, x_n \dots \}$ is one
such that $x_1$ is not a zero divisor in $\Lambda$ and $x_{i+1}$ is
not a zero divisor of $\Lambda/(x_1,\dots,x_i)$ for all $i\geq 1$.
\end{defn}

\begin{thm}[Collapse Theorem]
\label{thm:EMSS Collapse}
Let
$$\begin{CD}
X \times_B Y @>>> Y \\
@VVV @VVpV \\
X @>f>> B \\
\end{CD}$$
be a pullback square of spaces, with $B$ simply connected, $p$ a
Serre fibration. Furthermore, assume that \\  (i) $p^*: H^*(B) \to
H^*(Y)$ is onto, \\  (ii) $\ker(p^*)$ is a Borel ideal of $H^*(B)$,
and \\  (iii) ${\rm Im} (f^*)$ generates a Borel ideal $J$ of
$H^*(X)$.\\
Then, the associated Eilenberg-Moore spectral sequence collapses at
the $E_2$-page and we have
$$
E_2^{*,*} = E_{\infty}^{*,*} \cong (H^*(X)/J) \otimes
\Lambda^*(u_1,\dots) \cong H^*(X \times_B Y)
$$
where $\Lambda^*(u_1,\dots)$ is an exterior algebra on generators
all of which lie on the $p=-1$ line of the spectral sequence.
\end{thm}
\begin{proof}
See  Theorem 3.1 in \cite{Smith1}.
\end{proof}

\begin{rem}
Note that in the above theorem, the final isomorphism is an
isomorphism of $k$-vector spaces in general due to lifting issues
when lifting the algebra structure over the associated filtrations.
We will come back to this in our specific example later.
\end{rem}

In our case we are interested in studying the cohomology of groups
$G(Q)$ given by central extensions
$$E(Q): 0 \to V \to G(Q) \to W \to 0$$
where $V$ and $W$ are elementary abelian. We have seen that the
category of such extensions is naturally equivalent to the category
of quadratic maps from $W$ to $V$. However it is basic in group
cohomology that the equivalence classes of extensions of this form
are uniquely specified by a cohomology class in $H^2(W,V)$. By the
representability of cohomology, $H^2(W,V)$ is itself isomorphic to
$[BW, BBV]$ where $[-,-]$ denotes homotopy classes of maps, $BG$
denotes the classifying space of a monoid $G$, and $BBV=K(V,2)$.

Under this correspondence we will denote the homotopy class
representing a specific quadratic form $Q: W \to V$ by $BQ: BW \to
BBV$. Thus we have a pullback square
$$
\begin{CD}
BG(Q) @>>> EBV \\
@VVV @VVpV \\
BW @>BQ>> BBV \\
\end{CD}
$$
for which we have an EM-spectral sequence
$$E_2^{*,*}=\Tor_{H^*(BBV)} (H^*(BW),\FF _2) \Rightarrow
H^*(G(Q), \FF _2).$$ In fact the same pullback square above exhibits
$BG(Q)$ as the homotopy fiber of the map $BQ$. We record these
observations:

\begin{pro} Let $V$, $W$ be elementary abelian 2-groups.
Then there is a bijective correspondence between equivalence classes
of quadratic forms $Q: W \to V$ and $[BW, BBV]$. If $BQ: BW \to BBV$
is the homotopy class of maps associated to a quadratic form $Q: W
\to V$ then $BG(Q)=HF(BQ)$ where $HF$ stands for homotopy fiber.
\end{pro}

Similarly, there is a natural equivalence between $n$-ary forms $W
\to V$ and homotopy classes of maps $[BW,B^nV]$ where $B^nV=K(V,n)$.
However in this case the homotopy fiber is not an Eilenberg-MacLane
space but has a two stage Postnikov tower.

It is well-known that if $\dim(V)=1$ then the characteristic element
$\kappa \in H^n(K(V,n))$ is a generator of mod-$2$ cohomology
$H^*(K(V,n))$ as a free module over the Steenrod algebra
$\mathcal{A}_2$. As an algebra $H^*(K(V,n))$ is a polynomial algebra
on all ``permissable Steenrod operation sequences'' on $\kappa$ of
excess less than $n$. The $\dim(V)
>1$ case follows similarly from K\" unneth's theorem. Using these
facts, one sees that conditions (i) and (ii) in
Theorem~\ref{thm:EMSS Collapse} hold trivially in this case. Thus it
remains to consider condition (iii).

When $Q$ is Bockstein closed, it is clear that  the ideal $J$
generated by ${\rm Im}(BQ)^*$ is exactly the ideal $I(Q)$ generated
by the components of $Q$. This is because the Bockstein is $Sq^1$
and all other Steenrod squares on these components are determined as
they have degree 2 and automatically lie in $I(Q)$. Thus in the case
of a Bockstein closed quadratic form, $J$ is a Borel ideal if and
only if $I(Q)$ is. This yields the following theorem immediately.

\begin{thm}[EM-SS collapse for $G(Q)$]
Let $E(Q):0 \to V \to G(Q) \to W \to 0$ be a Bockstein closed
extension associated to the quadratic map $Q: W \to V$. Assume that
$I(Q)$ is a Borel ideal, i.e., $\{ q_1, \dots, q_n \}$ is a regular
sequence in $H^*(BW)$. This holds for example when $Q$ is Bockstein
closed and $2$-power exact. Then, the EM-spectral sequence collapses
at the $E_2$-page and
$$E_2^{*,*} = E_{\infty}^{*,*} = A^*(Q) \otimes \Lambda^*(u_{ij}).$$
This lifts to give $H^*(G(Q),\mathbb{F}_2) \cong A^*(Q) \otimes
\Lambda^*(\hat{u}_{ij})$ as $\mathbb{F}_2$-vector spaces (not as
algebras). Here $A^*(Q)=H^*(W)/I(Q)$ and the $u_{ij}$ are exterior
Koszul generators corresponding to the polynomial subalgebra
$\mathbb{F}_2[\kappa_{ij} \mid  i=1,\dots, n,\ j \geq 0]$ of
$H^*(K(V,2))$. Specifically, $\kappa_{ij}=Sq^{2^j}\cdots Sq^2 Sq^1
\kappa_i$ were $\kappa_i$ is the $i$-th characteristic element in
$H^2(K(V,2))$. Thus $u_{ij}$ lies in bidegree $(-1, 2^{j+1}+1)$ in
the EM-spectral sequence.
\end{thm}

\begin{proof}
See the discussion above. The only thing that was not explained was
the selection of the exterior elements. This comes from the fact
that $(BQ)^*(\kappa_{i})=q_i$ and a change of rings isomorphism. For
details see \cite{Smith1}.
\end{proof}

The reader might have noticed that when we compare the result of the
EM-spectral sequence calculation with the result previously obtained
with the LHS-spectral sequence, they look different. The
LHS-spectral sequence gave under the same conditions that
$$H^*(G(Q),\mathbb{F}_2) \cong A^*(Q) \otimes \mathbb{F}_2[s_1,\dots,s_n]$$
as algebras. However the EM-spectral sequence above shows that
$$H^*(G(Q),\mathbb{F}_2) \cong A^*(Q) \otimes \Lambda^*(\hat{u}_{ij})$$
as vector spaces. The point is this last isomorphism is not of
algebras and the mentioned exterior algebra is isomorphic to the
polynomial algebra $\FF_2[s_1,\dots,s_n]$ as $\FF_2$-vector spaces.
Indeed note that $\hat{u}_{ij}$ lives in degree
$(2^{j+1}+1)-1=2^{j+1}$. The vector space isomorphism comes from
mapping $\hat{u}_{ij}$ to $s_i^{2^j}$ and extending to square free
powers of the $\hat{u}_{ij}$ as if making an algebra map. To see
this is indeed an isomorphism of graded vector spaces just note that
a power of $s_i$ say $s_i^n$ can be expressed uniquely as a (square
free) product of 2-power powers of $s_i$ using the 2-adic expansion
of $n$. For example $s_3^5=s_3^4s_3^1$ and hence would be associated
to $\hat{u}_{32}\hat{u}_{30}$ since $2^2=4$ and $2^0=1$. We record
this vector space isomorphism for reference.

\begin{pro}[Exterior-Polynomial vector space isomorphism]
\label{pro:ext=poly} If $$ \Lambda^* = \Lambda^*(\hat{u}_{ij} \mid
i=1,\dots,n , \ j \geq 0)$$ and $P^*=\mathbb{F}_2[s_1,\dots,s_n]$
with $|s_i|=2, |\hat{u}_{ij}|=2^{j+1}$ then $\Lambda^* \cong P^*$ as
graded vector spaces via the isomorphism that takes $\hat{u}_{ij}$
to $s_i^{2^j}$ extended to the canonical vector space basis of
$\Lambda^*$ in the obvious way.
\end{pro}

\section{Proof of Theorem \ref{thm:R=L}}
\label{sect:proof of the main theorem}

In this section we prove Theorem \ref{thm:R=L} using the
Eilenberg-Moore spectral sequence calculations given in the previous
section.

Before the proof, we first discuss the Steenrod algebra structure of
the EM-spectral sequence. For this purpose we will want to use the
bar resolution and not the smaller Koszul resolutions used
previously as it is more ``geometric''. This bar resolution has an
induced Steenrod module structure on the $p=-1$ and $p=0$ lines that
induces the Steenrod module structure on $H^*(G(Q))$. This is proven
as Corollary 4.4 in \cite{Smith1}. In fact, we will work out
representatives of the $u_{i0}$ in the bar resolution and work out a
formula for $Sq^1$ on them. Just for dimensional reasons, if we
throw the $u_{i0}$ into a vector $u_0$, one has $Sq^1(u_0)=Ru_0 +$
terms in lower filtration where $R \in \Hom(W, \End(V))$, or
intuitively, is a matrix with entries in $A^1(Q)$. By results in
\cite{Smith1}, this lifts to the identity $Sq^1(s)=Rs + $term in
$A^3(Q)$. Thus, if the calculation for $Sq^1(u_0)$ gives $R=L$, then
we will have the proof of Theorem \ref{thm:R=L}.

Let us first set out finding explicit representatives of the
$u_{i0}$ in the $E_1$ term of the bar resolution. Since we know they
survive to $E_{\infty}^{-1,3}$ we know they will have to be
permanent cycles.

In our case, the $p=-1$ line of the EM-spectral sequence (in the bar
resolution) is $H^*(BW) \otimes \bar{H}^*(BBV) \otimes k$. Let us
write $H^*(BW)=\mathbb{F}_2[x_1,\dots,x_n]$. Then $Sq^1(q_i)=
\Sigma_{j} L_{ij}q_j$ where $L_{ij} \in H^1(W)$ since our quadratic
form is Bockstein closed. Using the usual bar notation, let us
define elements $v_i = [1|\kappa_{i0}|1] + \Sigma_j
[L_{ij}|\kappa_j|1] \in E_1^{-1,3}$ where
$\kappa_{i0}=Sq^1(\kappa_i)$ and $\kappa_i$ is the $i$-th
characteristic element of $H^*(BBV)$. Let us compute $d_1(v_i)$. Let
$p: EBV \to BBV$ be the path loop fibration then
\begin{align*}
\begin{split}
d_1(v_i) &= ([1(BQ)^*(\kappa_{i0})|1] + [1|p^*(\kappa_{i0})1])
+ \Sigma_j ([L_{ij}(BQ)^*(\kappa_j)|1] + [L_{ij}|p^*(\kappa_j)1]) \\
&= [Sq^1(q_i)|1] + \Sigma_j [L_{ij}q_j | 1] \\
&= [Sq^1(q_i) + \Sigma_j L_{ij}q_j | 1] \\
&= 0.
\end{split}
\end{align*}
Since $v_i$ lies on the $p=-1$ line, all further differentials on
$v_i$ must vanish for dimensional reasons and so indeed the $v_i$
are permanent cycles. Also note that since $E_1^{-2,3}=0$ for
dimensional reasons, the vector space spanned by the $v_i$ is not
the image of $d_1$. Similarly for dimensional reasons it cannot be
the image of any $d_r$. (This is because one has to use $s$ tensors
of $\bar{H}^*(BBV)$ in the line $p=-s$ and hence the line vanishes
below $q=2s$.) Thus the space spanned by the $v_i$ embeds into
$E_{\infty}^{-1,3}$. However comparing with the Koszul resolutions
$E_{\infty}^{-1,3}$ term we see then that the span of the $v_i$ must
be the same as the span of the $u_{i0}$ in $E_{\infty}$. Thus,
without loss of generality we may assume (by changing our basis for
the span of the $s_i$) that $v_i$ represents $u_{i0}$. Thus it
remains to compute $Sq^1(v_i)$ to figure out $Sq^1(s_i)$ up to terms
of lower filtration.

We now carry out the computation of $Sq^1(v_i)$ using Larry Smith's
work showing that $Sq^1$ on $H^*(BW) \otimes \bar{H}^*(BBV) \otimes
\FF _2$ naturally induced by viewing that $p=-1$ line as $H^*(BW
\times BBV \times pt)$ is compatible with the $Sq^1$ action on the
cohomology of the total space $H^*(BG(Q), \FF _2)$. Note
$v_i=[1|\kappa_{i0}|1]+\Sigma_j [L_{ij}|\kappa_j|1]$. Here recall
that $[a|b|c]$ can be identified with the cross product $a \times b
\times c$ in $H^*(BW \times BBV \times pt )$. Using the Cartan
formula for Steenrod squares on cross products, we see that $Sq^1$
will operate like a derivation and so we get
\begin{align*}
\begin{split}
Sq^1(v_i) &= [1|Sq^1(\kappa_{i0})|1] + \Sigma_j ([Sq^1(L_{ij})|\kappa_j|1]
+ [L_{ij}|Sq^1(\kappa_j)|1]) \\
&=\Sigma_j ([L_{ij}^2|\kappa_j|1] + [L_{ij}|\kappa_{j0}|1])
\end{split}
\end{align*}
since $Sq^1(\kappa_{i0})=Sq^1(Sq^1(\kappa_i))=0$ and
$Sq^1(L_{ij})=L_{ij}^2$ since the degree of $L_{ij}$ is one. Since
the $p=-1$ line is a module over the $p=0$ line, we can write
$$[L_{ij}|\kappa_{j0}|1]=L_{ij}[1|\kappa_{j0}|1]=L_{ij}(v_j+\sum_k
[L_{jk}|\kappa_k|1]).$$ Plugging this into the above equations and
simplifying, one gets
$$Sq^1(v_i) = \sum_j L_{ij}v_j + (\sum_j L_{ij}^2 [1|\kappa_j|1]+
\sum_j \sum_k L_{ij}L_{jk} [1|\kappa_k|1]).$$ Let $v$ be the column
vector with $i$-th coordinate $v_i$, and let $L$ be the matrix with
$(i,j)$-entry $L_{ij}$ and $c$ be the column vector with $i$-th
coordinate $[1|\kappa_i|1]$. Then the above equation becomes
$$Sq^1(v)=Lv + (Sq^1(L)+L^2)c$$
where $Sq^1$ on a matrix just means apply it on each entry. In the
next theorem, we will show that the second term $(Sq^1(L)+L^2)c$ in
the above equation is zero in $E_2^{*,*}=E_{\infty}^{*,*}$ by
showing that it is a boundary under $d_1$ and this then shows that
$Sq^1(v)=Lv$ in $E_{\infty}^{*,*}$ which lifts to show $Sq^1(s) = Ls
+ \eta$ with $\eta$ a column matrix with entries in $A^3(Q)$.

\begin{thm}[Steenrod Structure]
\label{thm:Steenrod Structure} Let $Q: W \to V$ be a Bockstein
closed quadratic map whose components $q_i$ form a regular sequence
in $H^*(W)$ (for example if it is 2-power exact). Let $\dim(V)=n$,
then $H^*(G(Q)) \cong k[s_1,\dots,s_n] \otimes A^*(Q)$ as algebras.
If we write $\beta(q)=Sq^1(q)=Lq$ where $q$ is the column vector
with entries the components of $Q$ (we can write this as $Q$ is
Bockstein closed) then we have:
$$\beta(s)=Ls + \eta$$ where $\eta$ is a column vector with entries in
$A^3(Q)$ and $s$ is the column vector with entries the $s_i$. Since
$A^*(Q)$ is the image of $H^*(BW) \to H^*(BG(Q))$, this determines
the structure of $H^*(G(Q))$ over the Steenrod algebra completely as
all Steenrod operations can be determined from this information and
the axioms of the Steenrod algebra.
\end{thm}
\begin{proof}
In the paragraph before the statement of this theorem, it was shown
in the bar resolution that
$$Sq^1(v) = Lv + (Sq^1(L)+L^2)c$$ in $E_1^{-1,4}$ of the bar
resolution model of the Eilenberg-Moore spectral sequence.
Furthermore we showed that $v$ consists of permanent cycles that
survive and represent $s$. Thus to get the main formula of the
theorem it is sufficient to show $(Sq^1(L)+L^2)c$ is zero in
$E_2^{-1,4}=E_{\infty}^{-1,4}$ as this will give $$Sq^1(v)=Lv$$
which will lift to
$$Sq^1(s)=Ls + \eta$$ where $\eta$'s components have to live in lower
filtration, i.e., on $E_{\infty}^{0,3}=A^3(Q)$. This will prove the
theorem as the other statements have been proven in previous
sections. Thus it remains to show that $(Sq^1(L)+L^2)c$ is a
boundary under $d_1$.

Note that by the regularity of the sequence $q_1,\dots, q_n$, the
equality $[\beta(L)+L^2]q=0$ gives that $$\beta(L)+L^2=\sum_i T_i
q_i $$ for some scalar matrices $T_1,.., T_n$ (see Proposition 8.2
in \cite{PaYa}). Note that $[B(L)+L^2]q=0$ gives that for all $k$,
$$ \sum _i \sum _j  T_i (k, j) q_i q_j =0.$$
Hence, $T_i (k, i)=0$ for every $k$ and $i$, and $T_i (k, j)= T_j
(k, i)$ for every $1\leq i,j \leq n$.
Here we are using the fact that the set
$$\{q_iq_j \mid 1\leq i < j \leq n  \}$$ is linearly independent in
$H^*(W)$ as the $q_i$ are algebraically independent since they form
a regular sequence. Let us call this common entry $$a_k (i,j)=T_i
(k, j)= T_j (k, i).$$ The $a_k$ are then (skew) symmetric scalar
matrices with zeros down the diagonal as $a_k(i,i)=T_i(k,i)=0$.
Since we are in characteristic 2, we can view these as skew
symmetric matrices which will be useful in what follows.

Let $$\alpha_k= \sum _{i,j} a_k (i,j) [1|\kappa_{i}|\kappa_j|1]$$ in
$E_1 ^{-2, 4}$. Then, $$d_1(\alpha_k)= \sum _{i,j} a_k (i,j) ([q_i|
\kappa_{j}|1] + [1|\kappa_i\kappa_j|1]) =\sum _{i,j} a_k (i,j)
[q_i|\kappa_j|1]$$ as the second term in the first sum is zero due
to the skew-symmetry of the matrices $a_k$. Thus using that the
$p=-1$ line is a module over the $p=0$ line with an action that is
purely in the left slot we get
$$d_1(\alpha_k) = \sum_{ij} a_k(i,j)q_i[1|\kappa_j|1].$$
Using that $\sum_i a_k(i,j)q_i=\sum_i T_i(k,j)q_i$ is the
$(k,j)$-entry of $B(L)+L^2$ we get:
$$ d_1(\alpha_k) = \sum_j(B(L)+L^2)_{k,j} [1|\kappa_j|1].$$
Writing $c$ to be the column vector with components $[1|\kappa_j|1]$
as before and $\alpha$ to be the column vector with entries the
$\alpha_k$, then this reads as
$$d_1(\alpha) = (B(L)+L^2)c$$
which shows that indeed the desired term is a $d_1$ boundary and hence
completes the proof of the theorem.
\end{proof}

\section{An obstruction class for uniform double lifting}
\label{sect:uniformdoublelifting}

Let  $E(Q): 0 \to V \to G(Q) \to W \to 0 $ be a Bockstein closed
$2$-power exact extension with associated quadratic map $Q: W \to
V$. Recall that this means $\dim W=\dim V =n$, and the quadratic map
$Q$ is Frattini and effective. Note that in this case $k$-invariants
$q_1,\dots, q_n$ form a regular sequence, and as a consequence there
is a unique $L \in \Hom (W, \End (V))$ such that $\beta(q)=Lq$ (see
Proposition 8.1 of \cite{PaYa}).

First we make a simple observation about $L$. Note that applying the
Bockstein operator to the equation $\beta(q)=Lq$ gives
$$0 = \beta(L)q + L \beta (q)=[\beta (L)+L^2 ] q.$$
Since $q_1, \dots , q_n$ forms a regular sequence, the entries of
$\beta (L)+L^2 $ must be linear combinations of components of $q$
(see Proposition 8.2 in \cite{PaYa}). This means that linear map $L
: W \to \End(V)$ extends to a morphism $\rho_L : Q \to Q_{\gl (V)} $
in $\bf Quad$. So, $L$ defines a representation of $Q$. Let us
denote this representation also with $L$.

Now, we consider the cohomology ring of the group $G(Q)$. By Theorem
\ref{thm:mod 2 cohomology}, we have
$$H^*(G(Q);\FF _2) \cong \FF _2[s_1,\dots,s_n] \otimes A^*(Q)$$
as graded algebras, where $\deg(s_i)=2$ for $i=1,\dots, n$. In the
previous section, we calculated the images of Bockstein operator on
$s_i$'s. We obtained that
$$ \beta (s)=Ls +\eta $$
where $\eta$ is a column vector whose entries are in $H^3(W, \FF _2
)$. Applying the Bockstein operator again, we get
$$ 0=\beta (L)s+L \beta (s) + \beta (\eta)= [\beta (L)+L^2 ]s
+[\beta (\eta)+L \eta].$$ Since, we already know that $\beta
(L)+L^2=0$ in $A^*(Q)$, we obtain that
$$\beta(\eta)+L \eta =0$$ in $A(Q)^*$.  This shows that shows that $\eta $ is a
$3$-dimensional cocycle in the chain complex $C^* (Q, L)$. So, it
defines a three dimensional cohomology class $[\eta]\in H^3 (Q, L)$.
We will now show that this class is an obstruction class for
uniformly lifting the extension $E(Q)$ twice.

Let $E' : 0 \to V \to \Gamma  \to G\to 0$ denote the extension with
extension class $s \in H^2 (G, V)$ where $G=G(Q)$. Since $\res ^{G}
_V s_i=t_i ^2 $ for all $i$, the restriction of this extension to $V
\subseteq G$ gives the extension $0 \to V \to M \to V \to 0 $ where
$M \cong (\ZZ /4)^n$ as an abelian group. Let $0 \to V \to M \to V
\to 0$ be the $\ZZ /4 W$-lattice whose logarithm is $L : W \to \End
(V)$. By this we mean that the representation for $M$ is of the form
$w \to I+2L(w)$ mod $4$ (see Lemma 4.3 \cite{PaYa} for details). By
Theorem 4.6 in \cite{PaYa}, $E'$ lifts to an extension
$$E'' : 0 \to M \to \widetilde \Gamma \to G\to 0 $$ if and only if
$\beta (s)=Ls$. So, $\eta =\beta (s)+Ls$ is the obstruction for such
a lifting. But, the extension $E'$ is determined by our choices of
generators $s_i$. We can replace each $s_i$ with
$$s_i ' = s_i + \xi_i$$
for some $\xi _i \in H^1(W, \FF _2)$ and the resulting extension $E'$ will still satisfy
the above condition. It may happen that for this new extension $E'$,
the obstruction for lifting it to $E''$ is zero. Let us calculate
the new obstruction
$$\eta'=\beta (s')+Ls'=\beta (s+\xi)+L(s+\xi)=\eta+
(\beta(\xi)+L\xi).$$ Note that the class $\beta(\xi)+L\xi$ is equal
to $\delta (\xi)$ where $\delta$ is the boundary map of the chain
complex $C^* (Q, L)$. So, we proved that in the set of all possible
extensions $E' : 0 \to V \to \Gamma \to G \to 0$ whose restriction
to $V \subseteq G$ gives the extension $0 \to V \to M \to V \to 0 $,
there is at least one extension $E'$ that lifts to $E''$ of the form
$$E'' : 0 \to M \to \widetilde \Gamma \to G\to 0 $$
if and only if $\eta$ is a coboundary, i.e., the cohomology class
$[\eta]=0$ in $H^3(Q, L)$. So, the proof of Theorem
\ref{thm:doublelifting} is complete.

\section{The Bockstein spectral sequence}
\label{sect:BockSpectralSeq}

Fix a prime $p$ and let $\FF _p$ be the field with $p$ elements. The
Bockstein spectral sequence is a well-known technique of relating
the mod-$p$ cohomology of a space with its integral cohomology. It
arises from the Massey exact triple
$$\xymatrix{H^*(X;\mathbb{Z})\ar[rr]^{\cdot p} &&H^*(X;\mathbb{Z})\ar[dl]^
{\varphi}\\ &B_1^*=H(X,\FF_p)\ar[ul]^{\hat{\beta}}&}$$ associated
to the sequence of coefficients $0 \to \mathbb{Z} \to \mathbb{Z} \to
\FF_p \to 0$. Here $\hat{\beta}$ is the integral Bockstein operator and $\varphi$ is the map induced by mod $p$ reduction. The
composition $\varphi \circ \hat{\beta}$ yields the differential
given by the standard Bockstein $\beta$ on $H^*(X;\FF_p)$. If we
denote the cohomology of this differential by $B_2^*$ then the
Massey exact triple above yields an exact triple
$$\xymatrix{pH^*(X;\mathbb{Z})\ar[rr]^{\cdot p} &&pH^*(X;\mathbb{Z})
\ar[dl]^{\varphi_2}\\&B_2^*\ar[ul]^{\hat{\beta}_2}& }$$ and by
iteration we get Massey exact triples of the form
$$\xymatrix{p^rH^*(X;\mathbb{Z})\ar[rr]^{\cdot p} & & p^rH^*(X;\mathbb{Z})
\ar[dl]^{\varphi_{r+1}}\\&B_{r+1}^*\ar[ul]^{\hat{\beta}_{r+1}}&
}$$ where $\varphi_{r+1} \circ \hat{\beta}_{r+1}=\beta_{r+1}$
gives a differential on $B_{r+1}^*$ called the $(r+1)$-st higher
Bockstein. The collection of the Massey exact triangles above is
referred to as the {\it Bockstein spectral sequence} for $X$ at the
prime $p$.

Furthermore when $X$ is the classifying space of a finite $p$-group,
it is known that $B^*_{\infty}=H^*(pt ,\FF_p)$ is the limit of the
$B_r^*$. Note if $s > 0$ and $\exp(H^s(X;\mathbb{Z}))=p^{r+1}$, then
the multiplication by $p$ in this last diagram is zero. Hence
$\varphi _{r+1}$ gives an injection $0 \neq p^rH^s(X;\mathbb{Z})
\to B_{r+1}^s$ and hence $B_{r+1}^s \neq 0$. On the other hand it is
clear that $B_{r+2}^s=0$ from the next Massey triangle. Thus it is
not hard to see that $\exp(H^s(X;\mathbb{Z}))=p^{r+1}$ if and only
if $B_{r+1}^s \neq 0$ but $B_{r+2}^s = 0$. So, by seeing how quickly
any level of the Bockstein spectral sequence converges to zero, we
can determine the exponent of the integral cohomology of the
$p$-group at that level.

Now we have seen that if $Q: W \to V$ is a quadratic map associated
to a Bockstein closed 2-power exact sequence $E(Q): 0 \to V \to G(Q)
\to W \to 0$ which lifts uniformly twice then
$$H^*(G(Q);\mathbb{F}) \cong \mathbb{F}[s_1,\dots,s_n] \otimes A^*(Q)$$
where $n=\dim(W)$. As $Q$ is Bockstein closed, we may write
$\beta{q}=Lq$ for $L \in \Hom(W,\End(V))$ as usual and we have seen
in this case we can choose $s$ so that $\beta{s}=Ls$ also. Thus
$(B_1^*(G(Q)),\beta_1)$ is completely known from the matrix $L$
given from the Bockstein closed quadratic form $Q$.

To compute $B_2^*$, we decompose the polynomial algebra
$\mathbb{F}[s_1,\dots,s_n]=\oplus_{i=0}^{\infty} S^i$ into its
homogeneous components $S^i$. We then get a decomposition of
differential algebras:
$$(B_1^*(G(Q)),\beta_1) = \oplus_{i=0}^{\infty} (S^i \otimes A^*(Q), \beta)$$
and thus $$ B_2^*(G(Q)) = \oplus_{i=0}^{\infty} H^*(S^i \otimes
A^*(Q),\beta).$$ Note that $B_2^*(G(Q))$ contains the information
for the distribution of all the higher torsion (of exponent greater
than $p$) for $H^*(G(Q);\mathbb{Z}).$ Note that $S^0=\FF_2$, so the
term corresponding to $i=0$ is just the cohomology of $A^*(Q)$ under
$\beta$ which is $H^*(Q;\FF_2)$.

Note that for the summand coming from $S^1$, we have a differential
given by $\delta(f)=\beta{f}+L(f)$ and so the cohomology of that
term corresponds to $H^*(Q ; L)$ where $L$ denotes the $Q$-module
associated to the matrix $L$. However since the elements of $S^1$
are degree 2 elements, the contribution of $H^*(Q;L)$ will be in
degree $*+2$. Since $S^i$ is just the $i$-fold symmetric product of
$S^1$, the complexes for these terms are derived from those for
$S^1$, and we will denote the corresponding matrix and $Q$-module by
$\Sym^i (L)$. Note also that since the elements of $S^i$ have degree
$2i$, their contribution will be shifted in degree by $2i$. By these
observations we can conclude the following:

\begin{thm}
Let $G(Q)$ be a group associated to a Bockstein closed, 2-power
exact quadratic form with $\beta(q)=Lq$. Furthermore suppose $Q: W
\to V$ can be lifted uniformly twice. Then
$$H^*(G(Q);\FF_2) \cong \mathbb{F}_2[s_1,\dots,s_n] \otimes
A^*(Q)$$ where $n=\dim(W)$ and the higher torsion is given by
$$B_2^* \cong \oplus_{i=0}^{\infty} H^{*-2i}(Q,\Sym^i(L)).$$
Thus, the higher torsion in $G(Q)$ is computable from the cohomology
of the quadratic form $Q$ in coefficients given by various symmetric
powers of the defining module $L$.
\end{thm}

Note that for each $j\geq 0$, the $j$-th term $B_2^j$ in the above
formula is given by a finite direct sum since there are only a
finite number of nonzero coefficient modules $\Sym^i(L)$ in the sum.
Furthermore, note that the cochain complex used to compute
$H^*(Q,\Sym^i(L))$ has length always given by the length of the
cochain complex $A^*(Q)$ which is finite. Thus the equation above
does yield an algorithm to compute $B_2^j$ for every $j\geq 0$.

\acknowledgement{We thank the referee for careful reading of the paper and for many helpful comments.}

\bigskip

\noindent
Dept. of Mathematics \\
University of Rochester, \\
Rochester, NY 14627 U.S.A. \\
E-mail address: jonpak@math.rochester.edu \\

\bigskip

\noindent
Dept. of Mathematics\\
Bilkent University\\
Ankara, 06800, Turkey. \\
E-mail address: yalcine@fen.bilkent.edu.tr \\


\begin{thebibliography}{EMG}

\bibitem{Adem}
A.~Adem, {\it Cohomological exponents of $\ZZ G$-lattices}, J. Pure
Appl. Algebra {\bf 58} (1989), 1--5.

\bibitem{BensonII}
D. J. ~Benson, \emph{Representations and cohomology II:
 Cohomology of groups and modules},
Cambridge studies in advanced mathematics \textbf{31} (1998).

\bibitem{BrPa}
W.~Browder and J.~Pakianathan, {\it Cohomology of uniformly powerful
$p$-groups}, Trans. Amer. Math. Soc.  {\bf 352}  (2000), 2659--2688.

\bibitem{MiSy}
P.~A.~Minh and P.~Symonds, {\it The cohomology of pro-$p$ groups
with a powerfully embedded subgroup}, J. Pure Appl. Algebra {\bf 189} (2004), 221--246.

\bibitem{Pak}
J.~Pakianathan, {\it Exponents and the cohomology of finite groups},
Proc. Amer. Math. Soc. {\bf 128} (1999), 1893--1897.

\bibitem{PaRo}
J.~Pakianathan and N.~Rogers, {\it Higher torsion in $p$-groups, Casimir operators and the classifying spectral sequence of a Lie algebra}, preprint, 2010.

\bibitem{PaYa}
J.~Pakianathan and E.~Yal\c c\i n, {\it Quadratic maps and Bockstein
closed group extensions}, Trans. Amer. Math. Soc. {\bf 359} (2007),
6079--6110.

\bibitem{Rector} D.~Rector, {\it Steenrod operations in the Eilenberg-Moore
spectral sequence}, Comment. Math. Helv. {\bf 45} (1970), 540--552.

\bibitem{Rusin}
D.~Rusin, {\it The mod $2$ cohomology of metacyclic $2$-groups}, J.
Pure Appl. Algebra {\bf 44} (1987), 315--327.

\bibitem{Smith1}
L.~Smith, {\it Homological Algebra and the Eilenberg-Moore Spectral
Sequence},  Trans. Amer. Math. Soc. {\bf 129} (1967), 58--93.

\bibitem{Smith2}
L.~Smith, {\it On the K\" unneth Theorem I}, Math. Z. {\bf 116}
(1970), 94--140.

\bibitem{Smith3}
L.~Smith, {\it Polynomial invariants of finite groups}, A. K.
Peters, London (1995).

\end{thebibliography}
\end{document}